\documentclass[12pt,reqno,oneside]{amsart}
\usepackage{amsmath}
\usepackage{amssymb}
\usepackage{amsthm}
\usepackage{graphics}
\usepackage{eucal}
\usepackage{bm}
\usepackage[colorlinks]{hyperref}
\usepackage[capitalize,nameinlink,noabbrev]{cleveref}
\usepackage{mathrsfs}
\usepackage[normalem]{ulem}
\usepackage{color}
\usepackage[dvipsnames]{xcolor}
\usepackage{mathtools}
\usepackage{geometry}
\newtheorem{theorem}{Theorem}
\newtheorem{definition}{Definition}[section]
\newtheorem{proposition}{Proposition}[section]
\newtheorem{lemma}{Lemma}[section]
\newtheorem{corollary}{Corollary}[section]

\newtheorem{remark}{Remark}[section]
\newtheorem{problem}{Problem}[section]

\newtheorem{dream theorem}{Dream theorem}[section]
\numberwithin{equation}{section}

\newcommand{\diag}{\operatorname{diag}}

\newcommand{\supp}{\operatorname{supp}}

\newcommand{\Bad}{\mathbf{Bad}}
\newcommand{\bq}{\mathbf{q}}
\newcommand{\bx}{\mathbf{x}}

\newcommand{\bu}{\mathbf{u}}

\newcommand{\cM}{\mathcal{M}}
\newcommand{\bp}{\mathbf{p}}
\newcommand{\bmi}{\mathbf{m}}

\newcommand{\bv}{\mathbf{v}}
\newcommand{\bw}{\mathbf{w}}

\newcommand{\f}{\mathbf{f}}

\newcommand{\ba}{\mathbf{a}}

\newcommand{\by}{\mathbf{y}}

\newcommand{\Q}{\mathbb{Q}}

\newcommand{\Z}{\mathbb{Z}}
\newcommand{\R}{\mathbb{R}}
\newcommand{\N}{\mathbb{N}}
\newcommand{\ta}{\mathbf{\boldsymbol\theta}}

\newcommand{\sd}[1]{{\color{RubineRed} \sf SD: (#1)}}

\usepackage{mathtools}
\usepackage{tikz-cd}
\usepackage{graphicx}

\newcommand{\blue}[1]{{\color{blue}#1}}

\begin{document}
\title{Winning and nullity of inhomogeneous bad}

\author{Shreyasi Datta}
\address[Shreyasi Datta]{Department of Mathematics, University of York, UK}
\email{shreyasi.datta@york.ac.uk}
\author{Liyang Shao}
\address{Department of Mathematics, University of California, Berkeley, CA 94720-3840, United States}
\email{liyang$_{-}$shao@berkeley.edu,  shaoliyang755@gmail.com}

\date{}

\begin{abstract}
    We prove the hyperplane absolute winning property of weighted inhomogeneous badly approximable vectors in $\mathbb{R}^d$. This answers a question by Beresnevich--Nesharim--Yang and extends the main result of [Geometric
and Functional Analysis, 31(1):1–33, 2021] to the inhomogeneous set-up. 

We also show for any nondegenerate curve and nondegenerate analytic manifold that almost every point is not weighted inhomogeneous badly approximable for any weight. This is achieved by duality and the quantitative nondivergence estimates from homogeneous dynamics motivated by [Acta Math. 231 (2023), 1-30], together with the methods from [arXiv:2307.10109].

\end{abstract}

\maketitle
\section{Introduction}
%By Dirichlet's theorem, every real number can be approximated by countably many rational numbers $p/q$ at a rate of $1/q^2.$ In addition, Khintchine proved that we can make 1 in the numerator arbitrarily close to 0 except for a set of Lebesgue measure zero, while this set is of Hausdorff dimension 1 by Jarnik. These numbers are called \textit{badly approximable numbers}, which stem from Diophantine Approximation. Badly approximable numbers have been studied for centuries, and the definition has been extended to various situations. One of those is simultaneous inhomogeneous badly approximable vectors in $\R^d,$ which is the main topic we are going to discuss in this paper.
The main objective of this paper is to study inhomogeneous weighted badly approximable vectors as we define below. The article has two parts that are independent of each other and use techniques of different sorts. In the first part, we study \textit{Hyperplane absolute winning} property (which, in particular, depicts \textit{largeness}) of the set of our interest and in the second part we show that when intersected with suitable manifolds, the intersection has Lebesgue measure zero.

First, we recall some definitions. A vector $\bw=(w_1,\cdots,w_d)\in\R^d$ is called a \emph{weight vector} if  $w_i\geq 0$ for any $1\leq i\leq d$ and $\sum_{i=1}^d w_i=1.$ When all $w_i$'s are the same, we call $\bw=(1/d,\cdots, 1/d)$ the \emph{standard weight}. Given a weight vector $\bw\in\R^d$, and $\ta=(\theta_i)_{i=1}^d\in\R^d$, the set of inhomogeneous weighted badly approximable vectors is defined as follows,
\begin{equation}
\label{def: simul_inho}
    \mathbf{Bad}_{\ta}(\mathbf{w}):=\{\mathbf{x}=(x_i)_{i=1}^d\in\mathbb{R}^d:\liminf_{q\in\mathbb{Z}\setminus\{0\}, \vert q\vert \to\infty}\max_{1\leq i\leq d}\vert qx_i-\theta_i\vert_{\mathbb{Z}}^{1/w_i}   \vert q\vert>0 \}.
\end{equation}
Here $\vert \cdot\vert_{\Z}$ denotes the distance from the nearest integer. The Lebesgue measure of this set is zero; see \cite{BDGW_null}, while its Hausdorff dimension is full; see \cite{BRV_survey}. When $\ta=\mathbf{0}$, we refer the vectors in the above set as weighted badly approximable vectors and denote it by $\Bad(\bw)$. 

In the first part of this paper, we study $\Bad_{\ta}(\bw)$ intersected with \textit{fractals} which leads to Theorem \ref{thm: main}. In the second part, we study $\Bad_{\ta}(\bw)$ intersected with curves yielding Theorem \ref{measure-zero-theorem}. The proofs of these two theorems are independent of each other and methods are different.

\subsection{Hyperplane absolute winning}
In $\ta=\mathbf{0}$ case, Schmidt in \cite{Sch_Conj} conjectured that $\Bad(\bw)\cap\Bad(\bw')\neq\emptyset$ for two different weights $\bw,\bw'$ in $\R^2.$ This longstanding conjecture was settled by a groundbreaking work of Badziahin, Pollington and Velani in \cite{BPV2011} where the authors showed a stronger result, that the intersection of finitely many $\Bad(\bw)$ has full Hausdorff dimension. More generally compared to Schmidt's conjecture, Kleinbock in 1998 conjectured that the set of weighted badly approximable vectors $\Bad(\bw)$  has a stronger property of winning in the sense of Schmidt’s game; see \cite{Kleinbock_Duke_98}. In dimension $2$, this conjecture was settled by An in \cite{An}. In higher dimensions this problem was investigated by many authors in the past twenty years; see \cite{AGGL, KL-2016, EGL, GY, LDM, BMS2017, LW, KW13, KW10, BFKRW, BFK, ET, Mos11, Jim2, Fish09, KW2005, Da, Lei2019}. Recently in \cite{BNY21}, Beresnevich, Nesharim, and Yang solved this conjecture for arbitrary dimensions. Moreover, they show that $\Bad(\bw)\subset\R^d$ is \text{Hyperplane absolute winning} (abbreviated as \emph{HAW}), which is stronger than winning in the sense of Schmidt's game. HAW is a variant of winning according to Schmidt's game, introduced in \cite[\S 2]{BFK}; see \cite[Definition 6]{BNY21} for definition. Thus \cite{BNY21} resolves Kleinbock's conjecture in \cite{Kleinbock_Duke_98} assertively. This remarkable new developments in \cite{BNY21, BNY22} are our motivation for our first result. 

It is natural to seek if the above-mentioned properties are true more generally for $\mathbf{Bad}_{\ta}(\mathbf{w})$, for any $\ta\in \R^d$, the inhomogeneous set-up, as suggested/studied in \cite{BNY22}, \cite{ET}, \cite[Section 3]{BV13} and \cite{Kleinbock_inhomo}. In \cite{BNY22}, Beresnevich, Nesharim and Yang asked if their results can be generalized in the inhomogeneous set-up. This was also previously asked by Kleinbock in \cite{Kleinbock_inhomo} for the standard weight case. To be specific, analogous problem of Kleinbock's conjecture \cite{Kleinbock_Duke_98} in the inhomogeneous set-up is as follows :
\begin{problem}\label{prob}
    Is it true that $\mathbf{Bad}_{\ta}(\mathbf{w})$ is winning for every weight $\bw\in \R^d$ and every $\ta\in\R^d$?
\end{problem}

For general $\ta$ and general $\bw$  less has been known hitherto. In \cite{ET} it was shown in the case of linear forms, that for standard weight $\mathbf{Bad}_{\ta}(1/d,\cdots, 1/d)$ is winning for any $\ta$. For general weight, the only known result towards this was due to An--Beresnevich--Velani in \cite{ABV}, where $\mathbf{Bad}_{\ta}(\bw)\subset \R^2$ was shown to be winning for any weight $\bw\in \R^2.$ In arbitrary dimensions, authors show that countable intersection of sets of inhomogeneous weighted badly approximable vectors has full Hausdorff dimension, a weaker conclusion than in Problem \ref{prob}; see \cite[Corollary 1.2]{DS_curve}. Thus beyond dimension $2$, Problem \ref{prob} remained open.

In this paper, we answer Problem \ref{prob} and extend \cite[Theorem 3]{BNY21} of Beresnevich, Nesharim, and Yang to the inhomogeneous set-up. More generally, we consider $\ta:\R^d\to\R^d, \ta(\bx)=(\theta_i(x_i))$ to be a function and generalize the definition in \eqref{def: simul_inho} in a natural way: $\bx\in \Bad_\ta(\bw)$ if $\liminf_{q\in\mathbb{Z}\setminus\{0\}, \vert q\vert \to\infty}\max_{1\leq i\leq d}\vert qx_i-\theta_i(x_i)\vert_{\mathbb{Z}}^{1/w_i}   \vert q\vert>0$. Our main theorem in this direction is as follows.

\begin{theorem}
\label{thm: main}
    Let $\ta:\R^d\to\R^d$, with $\ta(\bx):=(\theta_i(x_i))$, $\theta_i:\R\to \R$ Lipschitz function and $\bw$ be a weight in $\R^d$. Then 
    $\mathbf{Bad}_{\ta}(\mathbf{w})$ is hyperplane absolute winning, in particular, winning.
\end{theorem}

     %Theorem \ref{thm: main} also generalises \cite{ET} to the weighted set-up in case of vectors. 

The \textit{Hyperplane absolute winning} property brings out many nice sub-properties \cite{BFKRW} including invariance under countable intersection, diffeomorphisms, and the full dimension within the support of the \textit{Ahlfors regular absolutely decaying measures} (see \S \ref{sec2} for definition).
%It is well known that in a metric space,
%\begin{enumerate}\item A countable intersection of HAW sets is still HAW\item Moreover, if the space is a differential manifold, the image of a HAW set under $C^1$ diffeomorphisms is still HAW\end{enumerate}
As a consequence of Theorem \ref{thm: main} together with \cite[Lemma 1.3]{BNY22} we obtain:
\begin{corollary}
    Let $\{\bw_j\}_{j=1}^{\infty}$ be a sequence of weights in $\R^d$, and $\{\ta^j\}_{j=1}^{\infty}$ be a sequence in $\R^d$
    %where each $\ta^j=(\theta_i^j)_{i=1}^{d}\in\R^d$, 
    and $\{f_j\}_{j=1}^{\infty}$ be $C^1$-diffeomorphisms of $\R^d$. Then 
    \begin{equation*}
        \bigcap_{j=1}^{\infty}f_j(\Bad_{\ta^j}(\bw_j))
    \end{equation*}
    is hyperplane absolute winning. In particular, for every Ahlfors regular absolutely decaying measure $\mu$ on $\R^d,$ we have

    $$\dim( \bigcap_{j=1}^{\infty}f_j(\Bad_{\ta^j}(\bw_j))\bigcap\supp\mu)=\dim(\supp\mu),$$ where $\dim$ denotes the Hausdorff dimension.
\end{corollary}

The proof of Theorem \ref{thm: main} uses a method introduced in \cite{BNY21}, namely that it is enough to show that $\Bad_{\ta}(\bw)$ intersected with Alhfors regular absolutely decaying measures is \textit{Cantor winning}, due to a game coined in \cite{BHNS}. The goal is to get a \textit{Cantor winning strategy} to remove obstructions for a point to be in the respective bad set. Our idea is to introduce a new winning strategy that incorporates the strategy as in the homogeneous case \cite{BNY21}, together with removing obstruction for a point to be in the inhomogeneous set $\Bad_{\ta}({\bw})$. In order to do this, the main challenge is to find a way to merge the inhomogeneous constraints into the homogeneous winning strategy in \cite{BNY21}; a philosophy that was used in \cite{BV13, ABV, DS_curve}. In all of these works \cite{BV13, ABV, DS_curve}, either it was done in dimension $1$ or a weaker game was played. Our main innovation is to adopt this philosophy even in this stronger game set-up. It is worth pointing out that in the homogeneous case \cite{BNY21}, absolutely decaying  property of  measures is required in order to use \textit{quantitative nondivergence} estimates from \cite{KLW}. In our case, we also require this property crucially in order to show that the obstructions to be in $\Bad_{\ta}(\bw)$ can be legally removed. Moreover, we can take $\ta$ to be a function in Theorem \ref{thm: main} in higher dimensions, which requires a new idea. To tackle the inhomogeneous part as a function in higher dimensions, we divide the space into rectangles suitably at every step of the game. As the game continues, these rectangles become smaller in size. The idea is to treat the inhomogeneous function $\ta$ as a constant function locally, at every step.

%\textit{dangerous rectangles} from $\Bad(\bw)$ as 

    %The proof of Theorem \ref{thm: main}
   % uses the winning strategy developed in \cite{BNY21} together with a new strategy that we introduce here. The main challenge is to find a way to merge the inhomogeneous constraints into the homogeneous winning strategy in \cite{BNY21}; a philosophy that was used in \cite{BV13, ABV, DS_curve}. \blue{Thus in order to get a} \sd{Expand more} Another idea in the proof of Theorem \ref{thm: main} is how to tackle the inhomogeneous part as a function in higher dimensions. For this, as the game continues we divide the balls into rectangles suitably at every step of the game. 
    
    %Although in \cite{DS_curve}, similar ideas were explored in dimension $1$, in higher dimensions one needs to be more careful as we deal with a general weight.

\subsection{Inhomogeneous Bad is null}
We look into the measure-theoretic property of $\Bad_{\ta}(\bw)$ complementing Theorem \ref{thm: main}. In particular, we study $\Bad_{\ta}(\bw)$ intersected with a manifold. The problem of interest is as follows.
\begin{problem}\label{prob_null}
    Let $\ta\in\R^n$ and $\bw$ be any weight in $\R^n$ and $\mathcal{M}$ be a submanifold in $\R^n$. Under what condition on $\mathcal{M}$, almost every point in $\mathcal{M}$ is not in $\Bad_{\ta}(\bw)?$
\end{problem}
In case of $\ta=\mathbf{0}$, and $\bw$ being standard weight the set of badly approximable vectors inside \textit{nondegenerate} manifolds (see \S \ref{null-sec} for definition) was shown to be of measure zero in various works \cite{Khi25,BBKM02,B12,Shah09}. These works (except Khintchine's result \cite{Khi25}) rely on homogeneous dynamics. In a recent work \cite{BDGW_null}, for $\ta=\mathbf{0}$ and any weight $\bw$, it was shown that for any $C^2$ manifold with some mild condition on tangent, almost every point is not in $\Bad(\bw).$ The proof in \cite{BDGW_null} does not use any tools from homogeneous dynamics, rather relies on a new idea of establishing \textit{constant invariance}.

In the case of $\ta$ not necessarily $\mathbf{0}$, i.e., in the inhomogeneous set-up, none of the above approaches are yet fully developed and have natural obstructions. There are only limited results, and all follow from a stronger result of establishing the divergence part of a Khintchine type results; see \cite[\S 1]{BDGW_null} for discussion around this. The only known result toward this is due to \cite{BVV11,BVVZ21} in case of standard weight, in particular, they show almost every point in a nondegenerate curve is not in $\Bad_{\ta}(1/d,\cdots,1/d).$ 

Our next theorem settles Problem \ref{prob_null} for any nondegenerate curve and any nondegenerate analytic manifold.

\begin{theorem}\label{measure-zero-theorem}
    Let $\bw$ be a weight in $\R^n$, $\ta\in\R^n$ and $\cM$ be a nondegenerate curve in $\R^n$. Then almost every point on $\cM$ is not in $\mathbf{Bad}_{\ta}(\bw)$.
\end{theorem}
 As a consequence we get the following corollary.
 
\begin{corollary}\label{analytic}\footnote{We thank Victor Beresnevich for suggesting Corollary \ref{analytic} and that it follows from Theorem \ref{measure-zero-theorem}.}
    Let $\bw$ be a weight in $\R^n$, $\ta\in\R^n$ and $\cM$ be a nondegenerate analytic manifold in $\R^n$. Then almost every point on $\cM$ is not in $\mathbf{Bad}_{\ta}(\bw)$.
\end{corollary}
\begin{remark}
We show a slightly more general result for higher dimensional manifolds; see Theorem \ref{thm: more general} where analyticity can be dropped if we assume $\bw$ restrictively. Moreover, one can show the same conclusion as in Corollary \ref{analytic}, by putting an abstract condition on $\mathcal{M}$, namely that $\mathcal{F}$ as in \eqref{def: F} has full measure.
 \end{remark}

\begin{remark} It is possible extend Theorem \ref{measure-zero-theorem} from constant $\ta$ to a function $\ta$ using the same technique as in Theorem \ref{thm: main}. We leave it for an enthusiastic reader.
\end{remark}
Theorem \ref{measure-zero-theorem} relies on showing the constant invariance as in \cite{BDGW_null}, but the way it was achieved in the homogeneous case $\ta=\mathbf{0}$ in \cite{BDGW_null} fails at the very first step. In \cite{BDGW_null} constant invariance property is obtained via a projection trick which relies on Minkowski's convex body theorem. In the inhomogeneous set-up, due to the lack of such convex body theorem, it is not clear how the same method can be applied. Thus the inhomogeneous problem becomes difficult to achieve without extra condition on the manifold (and/or on the weight). The new idea in this paper is to combine techniques from \cite{BDGW_null} and homogeneous dynamics via \textit{duality} introduced in \cite{BY}. The combination of such ideas is the novelty of the proof. Also the proof of Theorem \ref{measure-zero-theorem}, unlike $\ta=\mathbf{0}$ in \cite{BDGW_null}, involves Quantitative nondivergence from \cite{BKM} crucially. %The main novelty of this proof relies on a combination of ideas from both \cite{BY} and \cite{BDGW_null}. In particular, we use of duality that was introduced in \cite{BY} and the constant invariance ideas in \cite{BDGW_null}.

%\begin{remark}
    %Note that in the case of curves, the above theorem does not have any restriction on weight. Also, note that even for standard weight Theorem \ref{measure-zero-theorem} is new.
%\end{remark}

\section{Preliminaries}
\label{sec2}
\subsection{Preliminaries for Theorem  \ref{thm: main}} Given a metric space $(X,d)$, any $S\subset X$ and $r>0,$ we denote the closed $r $ neighborhood of $S$ by
\begin{equation*}
    B(S,r):=\{x\in X,d(x,S)\leq r\}.
\end{equation*}
A closed ball $B(x_0,r):=\{x\in X,d(x,x_0)\leq r\}$ is defined by a fixed center $x_0$ and a radius $r>0.$  Also, we denote $\bx=(x_i)\in\R^d$ where $x_i\in\R$. Suppose $S=B_1\times \cdots\times B_d$, where each $B_i$ is an interval in $\R$, then we call $S$ to be a rectangle and by $cS$ we mean $c B_1\times \cdots \times cB_d$.
First, we recall some definitions.

Since each $\theta_i, i=1,\cdots, d,$ is a Lipschitz function in Theorem \ref{thm: main}, there exists $\mathfrak{m}>0$ such that $\vert\theta_i(x)-\theta_i(y)\vert\leq \mathfrak{m}\vert x-y\vert, i=1,\cdots, d, \forall x,y\in\R$.
\begin{definition}
\label{ahlfors}
    A Borel measure $\mu$ on $\R^d$ is called $\alpha$-Ahlfors regular if there exists $A,r_0>0$ such that for any ball $B(\bx,\rho)\subset\R^d$ where $\bx\in\supp\mu$ and $\rho\leq r_0$, we have $$
    A^{-1}\rho^{\alpha}\leq\mu(B(x,\rho))\leq A\rho^{\alpha}.
    $$
\end{definition}
\begin{definition}
    \label{def: absolutely decaying}
    A Borel measure $\mu$ on $\R^d$ is called absolutely decaying if there exists $D,\delta>0$ and $r_0>0$ such that for any $\bx\in\supp \mu,0<r\leq r_0$, any hyperplane $H\subset\R^d$ and $r'>0$, we have that
    \begin{equation}
    \label{absolutely decaying}
        \mu(B(H,r')\cap B(\bx,r))\leq D\left(\frac{r'}{r}\right)^{\delta}\mu(B(\bx,r)).
    \end{equation}
\end{definition}

\begin{definition}\label{cantor game}
    Let $X$ be a complete metric space. The $\gamma$-Cantor potential game is played by two players, say Alice and Bob, who take turns making their moves. Bob starts by choosing a parameter $0<\beta<1$, which is fixed throughout the game, and a ball $B_0\subset X$ of radius $r_0>0$. Subsequently for $n=0,1,2,\cdots$, first, Alice chooses collections $\mathcal{A}_{n+1,i}$ of at most $\beta^{-\gamma(i+1)}$ balls of radius $\beta^{n+1+i}r_0$ for every $i\geq 0$. Then, Bob chooses a ball $B_{n+1}$ of radius $\beta^{n+1}r_0$ which is contained in $B_n$ and disjoint from $\bigcup_{0\leq l\leq n}\bigcup_{A\in\mathcal{A}_{n+1-l,l}}A.$ If there is no such ball then Alice wins by default. Otherwise, the outcome of the game is the unique point in $\cap_{n\geq 0}B_n.$

    A set $S\subset X$ is called $\gamma$-Cantor winning if Alice has a strategy that ensures that she either wins by default or the outcome lies in $S.$ If $X$ is the support of an $\alpha$-Ahlfors regular measure, then $S\subset X$ is called Cantor winning if it is $\gamma$-Cantor winning for some $0\leq\gamma<\alpha.$
\end{definition}
We recall the following theorem.
\begin{theorem}{\cite[Theorem 21]{BNY21}}
\label{nonempty}
    Let $X$ be the support of an $\alpha$-Ahlfors regular measure and let $S\subset X$ be Cantor winning, then $S\neq \emptyset.$
\end{theorem}

We also recall the following transference lemma from \cite{Mahler1939}.
\begin{lemma}
\label{transfer}
    Let $T_0,\cdots, T_n>0$, $L_0,L_1,\cdots,L_n$ be a system of linear forms in variables $u_0,u_1,\cdots,u_n$ with real coefficients and determinants $d\neq 0$, and let $L_0',L_1',\cdots,L_n'$ be the transposed system of linear forms in variables $v_0,v_1,\cdots,v_n$, so that $\sum_{i=0}^nL_iL_i'=\sum_{i=0}^nu_iv_i$. Let $\iota^n=\frac{\prod_{i=0}^nT_i}{|d|}.$ Suppose there exists a non-zero integer point $(u_0,u_1,\cdots,u_n)$ such that
    \begin{equation}
        \label{system1}
        |L_i(u_0,u_1,\cdots,u_n)|\leq T_i,~~\forall~~ 0\leq i\leq n.
    \end{equation}
    Then there exists a non-zero integer point $(v_0,v_1,\cdots,v_n)$ such that
    \begin{equation}
        \label{system2}
        |L_0'(v_0,v_1,\cdots,v_n)|\leq n\iota/T_0\textit{ and }|L'_i(v_0,v_1,\cdots,v_n)|\leq \iota/T_i,~~\forall~~ 1\leq i\leq n.
    \end{equation}
\end{lemma}

Also, we recall \cite[Lemma 22]{BNY21},
\begin{lemma}
\label{cover}
    Let $\mu$ be an Ahlfors regular measure on $\R^d$, $A,\alpha,\rho_0$ be as in Definition \ref{ahlfors} and $S\subset\R^d$ be any subset. Suppose $0<r\leq r_0$ and $\mu(B(S,r))<\infty$. Then there exists a cover of $S\cap\supp\mu$ with balls of radius $3r$ of cardinality at most
    \begin{equation}
        \frac{A\mu(B(S,r))}{r^{\alpha}}.
    \end{equation}
\end{lemma}

For any $g\in \mathrm{SL}_{d+1}(\R)$, $g\Z^{d+1}$ is a lattice in $\R^{d+1}$ with covolume $1$. Thus, by the equivalence map $g\mathrm{SL}_{d+1}(\Z)$ to $g\Z^{d+1}$, we associate $X_{d+1}$ with the space of all lattices with covolume $1$. For every $\varepsilon>0$, let us define 
$$
K_{\varepsilon}:=\{\Lambda\in X_{d+1}~|~ \inf_{\bx\in \Lambda\setminus 0}\Vert \bx\Vert_2\geq \varepsilon\}. $$ Mahler's compactness criterion ensures that these are compact sets in $X_{n+1}.$ Here and elsewhere $\Vert \cdot\Vert_2$ is the Euclidean norm.
\subsection{Preliminaries for Theorem \ref{measure-zero-theorem}}

First, let us recall the following lemma from \cite{BDGW_null}. 
Fix an integer $n \geq 1$, and for each $1 \leq i \leq n$ let $(X_{i}, d_{i}, \mu_{i})$ be a metric space equipped with a $\sigma$-finite Borel regular measure $\mu_{i}$. Let $(X,d,\mu)$ be the product space with $X=\prod_{i=1}^{n}X_{i}$, $\mu=\mu_{1}\times\cdots\times\mu_n$ being the product measure, and
\begin{equation*}
d(\bx^{(1)},\bx^{(2)})=\max_{1 \leq i \leq n}d_{i}(x_i^{(1)},x_i^{(2)})\,, \qquad \text{where }\bx^{(j)}=(x_1^{(j)},\dots,x_n^{(j)})\,\,\text{for }j=1,2.
\end{equation*}

\begin{lemma}[\cite{BDGW_null} Lemma 5.7] \label{CI_product} 
Let $n\in\N$. For each $1\le j\le n$ let $(X_j,d_{j},\mu_j)$ be a metric measure space equipped with a $\sigma$-finite doubling Borel regular measure $\mu_{j}$. Let $X=\prod_{j=1}^{n}X_{j}$ be the corresponding product space, $d=\max_{1 \leq j \leq n}d_{j}$ be the corresponding metric, and $\mu=\prod_{j=1}^{n}\mu_{j}$ be the corresponding product measure. Let $(S_i)_{i\in\N}$ be a sequence of subsets of $\supp \mu$ and $(\bm\delta_{i})_{i \in \N}$ be a sequence of positive $n$-tuples $\bm\delta_{i}=(\delta^{(1)}_{i},\dots,\delta^{(n)}_{i})$ such that $\delta^{(j)}_{i} \to 0$ as $i \to \infty$ for each $1 \leq j \leq n$. Let
\begin{equation*}
\Delta_n(S_{i}, \bm\delta_{i})=\{\bx \in X: \, \exists \, \ba \in S_{i} \, \, \, d_{j}(a_{j}, x_{j}) < \delta_{i}^{(j)} \, \, \forall \, \, 1 \leq j \leq n\}\,,
\end{equation*}
where $\bx=(x_1,\dots,x_n)$ and $\ba=(a_1,\dots,a_n)$.
Then, for any $\bm C=(C_{1},\dots , C_{n})$ and $\bm c=(c_{1},\dots , c_{n})$ with $0<c_j\le C_{j}$ for each $1 \leq j \leq n$
\begin{equation}\label{vb25}
\mu\left( \limsup_{i \to \infty} \Delta_n(S_{i}, \bm C\bm \delta_{i}) \;\setminus\;\limsup_{i \to \infty} \Delta_n(S_{i}, \bm c\bm\delta_{i}) \right)=0\,,
\end{equation}
where $\bm c\bm\delta_j=(c_1\delta^{(1)}_i,\dots,c_n\delta^{(n)}_i)$ and similarly $\bm C\bm\delta_j=(C_1\delta^{(1)}_i,\dots,C_n\delta^{(n)}_i)$.
\end{lemma}

\section{Hyperplane absolute winning: Theorem \ref{thm: main}}\label{Sec3}
To prove Theorem \ref{thm: main}, it suffices to show the following theorem, by \cite[Section 2, Proposition 11]{BNY21}.
\begin{theorem}
\label{thm: twin main}
Let $\ta, \mathbf{w}$ be as in Theorem \ref{thm: main}. Suppose $\mu$ be a compactly supported $\alpha$-Ahlfors regular absolutely decaying measure on $\R^d$, then 
\begin{equation}
    \mathbf{Bad}_{\ta}(\mathbf{w})\bigcap\supp\mu\neq\emptyset.
\end{equation}
\end{theorem}
\subsection{Preparation for the proof of Theorem \ref{thm: main} }

In the view of Theorem \ref{nonempty}  to prove Theorem \ref{thm: twin main} (which is equivalent to Theorem \ref{thm: main}) it suffices to show that for any $\alpha$-Ahlfors regular absolutely decaying measure $\mu,$ $\Bad_{\ta}(\bw)\cap \supp\mu$ is $\alpha'$-Cantor winning for some $\alpha'<\alpha$. Thus, from now on we are playing a Cantor game as described on Definition \ref{cantor game}. Assume Bob chooses $\beta$ and ball $B_0\subset \supp\mu$ of radius $r_0$. In particular, $B_0,\cdots, B_n,n\in\mathbb N$  balls and $0<\beta<1$ are chosen by the player Bob as in Definition \ref{cantor game}. Note that we can assume without loss of generality that $\beta$ can be chosen very small (will be determined later) as Alice's strategy can be only genuine at steps multiples of some $M$ for some $M\geq 1,$ otherwise, Alice plays arbitrarily. Also, without loss of generality, we assume that $r_0<1.$

Let us define $b$ such that
\begin{equation}
\label{beta}
    0<b^{-(1+w_1)}:=\beta<1.
\end{equation}
Without loss of generality, for the proof of Theorem \ref{thm: main} we assume
\begin{equation}\label{order in w}
    w_1\geq w_2\geq\cdots\geq w_d>0,
\end{equation}
and denote $t\in\N$ such that
\begin{equation}\label{defn: t}
    w_1=\cdots=w_t>w_{t+1}.
\end{equation}
Formally define $$w_{d+1}:=0.$$
For $y\in\R$ and $\bx\in \R^d$, let us define,
\begin{equation}
\begin{aligned}
    &a_y:=\textbf{diag}\{b^y,b^{-w_1y},\cdots,b^{-w_dy}\},\\&d_y:=\textbf{diag}\{\beta^{\frac{ty}{d+1}},\underbrace{\beta^{-\frac{(d+1-t)y}{d+1}},\cdots,\beta^{-\frac{(d+1-t)y}{d+1}}}_{t\text{ many}},\underbrace{\beta^{\frac{ty}{d+1}},\cdots,\beta^{\frac{ty}{d+1}}}_{d-t \text{ many }}\},\\&u_\mathbf{x}:=\begin{pmatrix}
        1&\mathbf{x}^T\\
        0&\mathrm{I}_d
    \end{pmatrix},  \text{ where } \mathrm{I}_d \text{ is the identity matrix of size  } d\times d.
    \end{aligned}
\end{equation}

Let us recall the following definitions
\begin{equation}
    \label{def: s}
    s:=\max\left\{5,\left\lceil\frac{1+w_1}{w_1-w_{t+1}}\right\rceil,\left\lceil\frac{2(1+w_1)+1}{w_1}\right\rceil\right\},
\end{equation}

\begin{equation}
    \label{def: eta}
    \eta:=\min\left\{\frac{1}{4(d+1)},\frac{w_ds}{2d(1+w_1)},\frac{w_1s-2(1+w_1)}{2d(1+w_1)}\right\},
\end{equation}
as in \cite[Equation (63),(64)]{BNY21}. Note that we can get rid of $\frac{\alpha}{\gamma}$ in \cite[Equation(64)]{BNY21}, which is the definition of $\eta$, because we can asssume $\gamma$ in \cite[Theorem 23]{BNY21} sufficiently small such that $\frac{\alpha}{\gamma}>\eta$ as in \eqref{def: eta}.

\subsection{Recalling the game in \cite{BNY21}}When $\ta=\mathbf{0}$, the winning result was shown in \cite{BNY21}. We first recall the strategy used by Alice in \cite[p. 20-21]{BNY21}.
For $i\geq 0$, let
\begin{equation}
    A_{1,i}:=\bigcup_{n'\geq 0,l\geq\frac{n'+1}{s},n'+(s-1)l=i}\{\bx\in B_0:d_la_{n'+1+sl}u_{\bx}\Z^{d+1}\notin K_{\beta^{\eta l}}\},
\end{equation} where the union is taken over all possible $n', l\in\N$
and let $A_{1,i}=\emptyset$ if no such $n',l$ exists.

Next, for $n\geq 1$, following \cite[Equation (69)]{BNY21} let
\begin{equation}\label{def: original target set}
    A_{n+1,i}:=\{\bx\in B_n:d_la_{n+1+sl} u_{\bx}\Z^{d+1}\notin K_{\beta^{\eta l}}\},
\end{equation}
if $i=(s-1)l$ for some integer $0<l<\frac{n+1}{s}$; otherwise we define $A_{n+1,i}:=\emptyset$.
Then by \cite[Equation (70)]{BNY21} define
\begin{equation}
\Tilde{A}_{n+1,i}:=A_{n+1,i}\setminus\bigcup_{0\leq n'<n,i'\geq0}A_{n'+1,i'}.
\end{equation}
%Finally, %by \cite[Equations (66),(67),(74)]{BNY21},  \sd{Rewrite the following Lemma starting with this line}

%Finally, we recall that we get  \cite[Equation (74)]{BNY21} using \cite[Lemma 27]{BNY21} from the construction of $\tilde{A}_{n+1,i}$ in \cite[Equation 70]{BNY21}. For this crucial observation, see \cite[p. 23-24]{BNY21}. To summarize: there exist $\beta_0>0$ (determined by \cite[Equations (66),(67)]{BNY21}) and $0\leq\alpha'<\alpha$ (determined by \cite[Equation (65)]{BNY21}, such that if $B_n\nsubseteq\cup_{0\leq n'<n,i'\geq0} A_{n'+1,i'}$ holds for every $n\in\mathbb N$, then for any $0<\beta\leq\beta_0, n\in\N, i\geq 0$, there exists  a cover $\mathcal{A}_{n+1,i}$ for $\Tilde{A}_{n+1,i}$, consisting of balls of radius $\beta^{n+1+i}r_0$ of cardinality at most $\beta^{-\alpha'(i+1)}$. Moreover, we know by the construction of $\tilde{A}_{n+1,i}$, $$\bigcup_{0\leq n'<n, i\geq 0} {A}_{n'+1,i}=\bigcup_{0\leq n'<n, i\geq 0} \tilde{A}_{n'+1,i}.$$ Combining this discussion, we have the following lemma. 
%\begin{lemma}
%\label{borrowed}
%There exist $\beta_0>0$ and $0\leq\alpha'<\alpha$, such that if \begin{equation}\label{condition}
%B_n\cap\cup_{0\leq n'<n,i'\geq0} \tilde A_{n'+1,i'}=\emptyset
%\end{equation} holds for every $n\in\mathbb N$, then for any $0<\beta\leq\beta_0, n\in \N, i\geq 0$, there exists a cover $\mathcal{A}_{n+1,i}$ for $\Tilde{A}_{n+1,i}$, consisting of balls of radius $\beta^{n+1+i}r_0$ and $$\#\mathcal{A}_{n+1,i}\leq \beta^{-\alpha'(i+1)}.$$ %\sd{ \text{ should be } \beta_0}
%\end{lemma}

Finally, we recall that we get  \cite[Equation (74)]{BNY21} using \cite[Lemma 27]{BNY21} from the construction of $\tilde{A}_{n+1,i}$ in \cite[Equation 70]{BNY21}. For this crucial observation, see \cite[p. 23-24]{BNY21}. To summarize: there exist $\beta_0>0$ (determined by \cite[Equations (66),(67)]{BNY21}) and $0\leq\alpha'<\alpha$ (determined by \cite[Equation (65)]{BNY21}, such that any $0<\beta\leq\beta_0, n\in\N, i\geq 0$, there exists  a cover $\mathcal{A}_{n+1,i}$ for $\Tilde{A}_{n+1,i}$, consisting of balls of radius $\beta^{n+1+i}r_0$ of cardinality at most $\beta^{-\alpha'(i+1)}$. Combining the above facts from \cite{BNY21}, we have the following lemma. 
\begin{lemma}
\label{borrowed}
There exist $\beta_0>0$ and $0\leq\alpha'<\alpha$, such that for any $0<\beta\leq\beta_0, n\in \N, i\geq 0$, there exists a cover $\mathcal{A}_{n+1,i}$ for $\Tilde{A}_{n+1,i}$, consisting of balls of radius $\beta^{n+1+i}r_0$ and $$\#\mathcal{A}_{n+1,i}\leq \beta^{-\alpha'(i+1)}.$$ %\sd{ \text{ should be } \beta_0}
\end{lemma}

\subsection{The new strategy}
Let $0<\beta\leq \beta_0$, where $\beta_0$ is as in Lemma \ref{borrowed}. Furthermore, suppose \begin{equation}\label{R_0'} \beta<\min\left\{r_0, \frac{1}{1+\mathfrak{m}}\right\}.\end{equation} From now on, such an $\beta$ is fixed. Let us recall that $\alpha, A, D$ are parameters of the Ahlfors regular absolutely decaying measure $\mu$.

%\sd{ Bob starts the game by choosing $\beta$ which is fixed throughout the game. So how is that we can assune $\beta<\beta_0?$ It is without loss of generality as Alice's strategy can be described with $\beta^{M}, M<1$ and letting her play arbitrarily on every step which is not $1$ modulo $M.$ I added it in the beginning of previous section.}\ls{You are right!}
Let $i_0\in\N$ be the smallest integer such that 
\begin{equation}\label{def: i_0}\beta^{(i_0+1)/2}\leq (A^2D6^{2\delta+\alpha})^{-1/\delta}, \text{ and } s-1\nmid i_0\end{equation} where $A,\alpha$ are as in Definition \ref{ahlfors}, $D,\delta$ are as in Definition \ref{def: absolutely decaying} and $s$ be as in \eqref{def: s}. We can see from these Definitions \ref{ahlfors} and \ref{def: absolutely decaying} that we can assume $A,D>1.$ Thus the right hand side of the first inequality above is less than $1.$
%\sd{What is $\delta?$}

Let $c<\beta^3$ be a sufficiently small constant which will be determined in \eqref{def: c}. Then for $n\in\N$, we define, 
\begin{equation}
\label{def: V_n}
    \mathcal{V}_n:=\left\{\bv=\frac{\mathbf{p}}{m}\in\Q^d, cr_0^{-1}\beta^{-(n+1)}\leq m^{1+w_1}<cr_0^{-1}\beta^{-(n+2)}, m\geq 3c^{-1}\beta^{-1}r_0^{-1}\right\}.
\end{equation}
By the choice of $c$ small enough we will have in \eqref{empty} that $\mathcal{V}_0=\emptyset.$
%\sd{$$\mathcal{V}_{n}:=\{(\bq,p)\in\Z^{n+1}~|~ \beta_1^{-(n+1)}\ll_{i} \max_{i=1}^d \vert q_i\vert^{1/w_i}\ll_{i}\beta_1^{-(n+2)}\}$$, for some $0<\beta_1<1$. Also, denote 
%$$\mathcal{V}_{n,k}:=\mathcal{V}_{n}\cap\{\beta^{-(n+1)w_{k}}\ll \Vert \bq\Vert \ll \beta^{-(n+1)w_{k-1}}\}, k=1,\cdots, d+1$$}
%Thus by \eqref{R_0'} and $c<\beta^3$  \Q^d\setminus \text{fineitly many}=
Then we have
\begin{equation}
    \Q_0^d:==\bigcup_{\ba\in\Z^d}(\Q^d\cap\prod_{i=1}^d[a_i, a_i+1]\setminus \text{finitely many})=\bigcup_{n\geq 1}\mathcal{V}_n
    \label{eqn: union}
\end{equation}

%\ls{Here, do we mean $\Q_0^d:=\bigcup_{\ba\in\Z^d}(\Q^d\cap\prod_{i=1}^d[a_i, a_i+1]\setminus \text{finitely many})=\bigcup_{n\geq 1}\mathcal{V}_n$?} \sd{yes you are right. Fixed}

 Let $\tilde B_0$ be the smallest cube in $\R^d$ that contains $B_0$. Given $n\in \N$, we divide $\tilde B_0$ as a disjoint union of rectangles of side lengths $\frac{r_0 c \mathfrak{q}}{\mathfrak{m}\beta^{-\frac{(n+2)}{(1+w_1)}w_i}}$ in the $i$-th direction where $1\leq i\leq d$ where $\mathfrak{q}$ is defined in \eqref{def: q}. Let us denote the collection of such rectangles intersected with the ball $B_0$ to be $\{B_0^{n,j}\}_{j}$. There atmost $$n^\star:=\frac{2^d \mathfrak{m}^d\beta^{-\frac{(n+2)}{(1+w_1)}}}{c^d \mathfrak{q}^d}$$ many such rectangles.  Let us fix points in $B_0^{j,n}$ to be $\mathbf{b}^{j,n}=(b_1^{j,n},\cdots, b_d^{j,n})$ and denote $\theta_i(b_i^{j,n})=\theta_i^{j,n},$ and $\ta^{j,n}=(\theta_i^{j,n}).$ Note that 
$B_0=\bigcup_{j=1}^{n^\star} B_0^{j,n}.$

Define
\begin{equation}
\label{def: q}
    \mathfrak{q}:=\beta^{(i_0+1)/2}(A^2D2^{2\delta+\alpha+d}3^{\alpha})^{-1/\delta} .
\end{equation}

In particular,
\begin{equation}
    \mathfrak{q}<1,
    \label{q<1}
\end{equation}
since $A, D>1.$
For $\bv=\frac{\mathbf{p}}{m}\in\mathcal{V}_n$ and 

\begin{equation}
\label{def: inho_interval}\begin{aligned}
 &\Delta_{\ta}(\bv):=\left\{\bx\in B_0:\left|x_i-\frac{p_i+\theta_i(x_i)}{m}\right|<\frac{\mathfrak{q}c}{m^{1+w_i}},\forall~ 1\leq i\leq d\right\},\\
    &\Delta_{\ta}^{j,n}(\bv):=B^{j,n}_0\cap \Delta_{\ta}(\bv) \text{ for } 1\leq j\leq n^\star.
\end{aligned}\end{equation}
Also for $\bv=\frac{\mathbf{p}}{m}\in\mathcal{V}_n$ and $1\leq j\leq n^\star$, let us define 
\begin{equation}
\label{def: inho_interval_new}
    \tilde \Delta_{\ta}^{j,n}(\bv):=\left\{\bx\in B^{j,n}_0:\left|x_i-\frac{p_i+\theta_i^{j,n}}{m}\right|<\frac{\mathfrak{q}c}{m^{1+w_i}},\forall ~ 1\leq i\leq d\right\}.
\end{equation}
\begin{lemma}
    For every $\bv=\frac{\mathbf{p}}{m}\in\mathcal{V}_n$, $1\leq j\leq n^\star$, 
    \begin{equation}
        \tilde \Delta_{\ta}^{j,n}(\bv)\subset  2\Delta_{\ta}^{j,n}(\bv)\subset  4\tilde \Delta_{\ta}^{j,n}(\bv),
    \end{equation}where for  $k>0, k\Delta^{j,n}_\ta(\bv):=B_0^{j,n}\cap k\Delta_\ta(\bv)$ is defined as in Equation \eqref{def: inho_interval} except that $\mathfrak q$ is replaced by $k\mathfrak q$.
    \label{compare dangerous domain}
\end{lemma}
\begin{proof}
For any $x\in \tilde \Delta_{\ta}^{j,n}(\bv), $ where $\bv=\frac{\bp}{m}\in \mathcal{V}_n,$
$$\begin{aligned}
    \left|x_i-\frac{p_i+\theta_i(x_i)}{m}\right|<& \left|x_i-\frac{p_i+\theta_i^{j,n}}{m}\right|+ \frac{1}{m}\left\vert \theta_i(x_i)-\theta_i^{j,n}\right\vert\\
    & <\frac{\mathfrak{q}c}{m^{1+w_i}}+ \frac{ \mathfrak{q}r_0 c}{m\beta^{-\frac{(n+2)}{(1+w_1)}w_i}}\\
    & < \frac{\mathfrak{q}c}{m^{1+w_i}}+ \frac{\mathfrak{q}c r_0}{m^{1+w_i}}(cr_0^{-1}})^{\frac{w_i}{1+w_1} 
\end{aligned}$$
The last inequality follows since $\bv\in \mathcal{V}_n$, $\beta^{-\frac{(n+2)}{(1+w_1)}w_i}\geq m^{w_i}{(cr_0^{-1}})^{-\frac{w_i}{1+w_1}}$. Now the first inclusion follows since $c<1$ and $r_0<1.$
The last inclusion follows similarly.
\end{proof}
%\sd{$\bv=(\bq,p)$, then $\Delta_{\ta}(\bv):=\{\bx\in B_0~|~\vert \bq\cdot\bx+p+\theta\vert <\frac{\delta' c}{ \max_{i=1}^d \vert q_i\vert^{1/w_i}}\}$}

Note that from \eqref{eqn: union} and  the Definition  \eqref{def: inho_interval}, we get
$$
B_0\setminus \bigcup_{\bv\in \Q_0^d }\Delta_{\ta}(\bv)\subset \Bad_\ta(\bw).$$ Also since $B_0=\bigcup_{j=1}^{n^\star} B_0^{j,n}$ for $n\geq 1$, by  \eqref{eqn: union} and Definition \eqref{def: inho_interval}, we have $\bigcup_{\bv\in \Q_0^d}\Delta_{\ta}(\bv)= \bigcup_{n\geq 1}\bigcup_{\bv\in \mathcal{V}_n}\bigcup_{j=1}^{n^\star}\Delta_{\ta}^{j,n}(\bv).$
Combining the above two observations, we get  
\begin{equation}\label{eqn: danger_bad}
    B_0\setminus \bigcup_{n\geq 1}\bigcup_{\bv\in\mathcal{V}_n}\bigcup_{1\leq j\leq n^\star}\Delta^{j,n}_{\ta}(\bv)\subset \Bad_\ta(\bw).
    \end{equation}

%\sd{
%Note the following is true by $B_0=\bigcup_{j=1}^{n^\star} B_0^{j,n}$ 

    %\begin{equation}\label{eqn: danger_bad}
    %B_0\setminus \bigcup_{n\geq 1}\bigcup_{\bv\in\mathcal{V}_n}\bigcup_{1\leq j\leq n^\star}\Delta^{j,n}_{\ta}(\bv)\subset \Bad_\ta(\bw).
    %\end{equation}

%The above inclusion follows from \eqref{eqn: union} and  the Definition  \eqref{def: inho_interval},
%$$
%B_0\setminus \bigcup_{\bv\in \Q_0^d }\Delta_{\ta}(\bv)\subset \Bad_\ta(\bw), \bigcup_{\bv\in \Q_0^d}\Delta_{\ta}(\bv)= \bigcup_{n\geq 1}\bigcup_{\bv\in \mathcal{V}_n}\bigcup_{j=1}^{n^\star}\Delta_{\ta}^{j,n}(\bv).$$ 
%}
%\ls{In the last equation, do we mean $\Delta_{\ta}(\bv)=\bigcup_{n\geq }\bigcup_{j=1}^{n^\star}\Delta_{\ta}^{j,n}(\bv)$? } \sd{I meant the above, earlier one extra union over $\mathcal{V}_n$ was not there}

%Define \begin{equation}\Tilde{A}_{n+1,i_0}':=\{\bx\in B_n:\exists~ \bv\in\mathcal{V}_n,\bx\in\Delta_{\ta}(\bv)\}\label{excluded}.\end{equation}

Define \begin{equation}\Tilde{A}_{n+1,i_0}':=\{\bx\in B_n:\exists~ \bv\in\mathcal{V}_n, \text{ and } 1\leq j\leq n^\star, \bx\in\Delta^{j,n}_{\ta}(\bv)\}\setminus\bigcup_{0\leq n'<n,i'\geq0}A_{n'+1,i'}\label{excluded}.\end{equation}
%and for $n\geq 0,i\neq i_0$, define
%\begin{equation}
   % \Tilde{A}_{n+1,i}':=\Tilde{A}_{n+1,i}.
%\end{equation}
%\sd{ I think we don't have to define $\Tilde{A}_{n+1,i}'$ when $i\neq i_0.$}
Now we state the following key lemma which we prove in \S \ref{Proofcantorwin}.
%\begin{lemma}\label{cantorwinning}  Let $i_0$ be as in \eqref{def: i_0}.
     %There exists $0\leq\gamma<\alpha$ and $c>0$ (which is used to define $\Tilde{A}_{n+1,i_0}'$), such  that if $$ B_n\cap\cup_{0\leq n'<n,i'\geq0} \tilde A_{n'+1,i'}=\emptyset \quad \forall n\geq 0,$$ then there exists a cover $ \mathcal{A}'_{n+1,i_0}$ for $\Tilde{A}'_{n+1,i_0}$ consisting of balls of radius $\beta^{n+1+i_0}r_0$ 
 %\begin{equation}
 %\label{goal}
    %\# \mathcal{A}'_{n+1,i_0} \leq \beta^{-\gamma(i_0+1)}.
 %\end{equation}
 %\end{lemma}
 \begin{lemma}\label{cantorwinning}  Let $i_0$ be as in \eqref{def: i_0}.
     There exists $0\leq\gamma<\alpha$ and $0<c<\beta^3$ (which is used to define $\Tilde{A}_{n+1,i_0}'$), such  that there exists a cover $ \mathcal{A}'_{n+1,i_0}$ for $\Tilde{A}'_{n+1,i_0}$ consisting of balls of radius $\beta^{n+1+i_0}r_0$ 
 \begin{equation}
 \label{goal}
    \# \mathcal{A}'_{n+1,i_0} \leq \beta^{-\gamma(i_0+1)}.
 \end{equation}
 \end{lemma}
 
 %\ls{With the change of definition in \eqref{excluded'}, there is no need to include the previous condition as well.}

%\blue{ We note that in the above lemma, the hypothesis is the same as in Lemma \ref{borrowed}, but the conclusion is about covering of the new sets $\tilde{A}'_{n+1,i_0}$ that we defined before.}

 \subsubsection{Proof of Theorem \ref{thm: main} modulo Lemma \ref{cantorwinning}} %To show \cref{thm: main}, by \cref{thm: twin main} and \cref{nonempty}, it suffices to show that $\mathbf{Bad}_{\ta}(\bw)$ is $\eta-$Cantor winning for some $0\leq\eta<\alpha$.
 We show $\mathbf{Bad}_{\ta}(\bw)$ is $\eta$-Cantor winning when $\eta=\min\{\gamma,\alpha'\}$, where $\alpha'$ is as in Lemma \ref{borrowed} and $\gamma$ is as in Lemma \ref{cantorwinning}, thus completing the proof of Theorem \ref{thm: main} in the view of Theorem \ref{thm: twin main} and Theorem \ref{nonempty}.
 
 %The strategy for Alice to win the game goes as follows inductively: In round $n$, Alice chooses the cover $\mathcal A'_{n+1,i},$ where for $i=i_0,$ we let $\mathcal A_{n+1,i_0}$ be as in \cref{cantorwinning}, while for $i\neq i_0,$ we let $\mathcal A_{n+1,i}'=\mathcal A_{n+1,i}$ be as in \cref{borrowed}.

Given Bob's moves $\beta, B_0,\cdots, B_n$ Alice chooses a sequence of $\{\mathcal{A}'_{n+1, i}\}_{i\geq 0},$ as follows:
\begin{equation}\label{eqn: strategy}
    \mathcal A'_{n+1,i}=\begin{cases}\begin{aligned}
& \mathcal{A}_{1,i} \text{ as in Lemma } \ref{borrowed} \text{ when } n=0, i\geq 0,\\ & \mathcal A'_{n+1,i_0} \text{ as in Lemma } \ref{cantorwinning} \text{ when } n\geq1, i= i_0, \\
\vspace{3 mm}
& \mathcal A_{n+1, i} \text{ as in Lemma } \ref{borrowed} \text{ when }n\geq1, i\neq i_0.

    \end{aligned}\end{cases}
\end{equation}

%\ls{Similarly, with the change of definition in \eqref{excluded'}, there is no need to justify any assumption}

 %Before we move on, we have to show the existence of such cover by showing that the assumption needed for \cref{cantorwinning} and \cref{borrowed} holds by induction. 
 %When $n=0, B_n\cap\cup_{0\leq n'<n,i'\geq0} \tilde A_{n'+1,i'}=\emptyset$ clearly holds because $\cup_{0\leq n'<n,i'\geq0} \tilde A_{n'+1,i'}=\emptyset$. Now if \begin{equation}
     %B_n\cap\cup_{0\leq n'<n,i'\geq0} \tilde A_{n'+1,i'}=\emptyset,\forall 0\leq n<N
 %\end{equation} then  
%since \begin{equation}
   % B_N\subset B_n,\forall 0\leq n<N
%\end{equation} in the game, it suffices to show that $B_N\cap\cup_{i'\geq0}\tilde A_{N,i'}=\emptyset$ to complete the induction $s-1\not|i_0$ and $\Tilde{A}_{n+1,i_0}=\emptyset\subset\Tilde{A}_{n+1,i_0}'$ by definition, we have 
%\begin{equation}
    %\label{containing}\Tilde{A}_{n+1,i}\subset\Tilde{A}_{n+1,i}',\forall N>n\geq0,i\geq0.
%\end{equation}

%As a consequence, $\forall \bx\in B_N,\bx\notin\bigcup_{n'< N,i'\geq 0}\Tilde{A}_{n'+1,i'}'\supset\bigcup_{n'< N,i'\geq 0}\Tilde{A}_{n'+1,i'},$ so $B_N\cap\cup_{0\leq n'<N,i'\geq0} \tilde A_{n'+1,i'}=\emptyset$ and the induction is completed. \ls{To be fixed later}

The strategy is legal since $\eta<\alpha$, and $\# \mathcal{A}_{n+1,i}'\leq \beta^{-\eta(i+1)}$ because of Lemma
\ref{cantorwinning} and Lemma \ref{borrowed}.

%Finally, for any $n\geq0,i\geq0$,
%\begin{equation}
    %\text{let }\mathcal{A}_{n+1,i}'\text{ be an efficient cover of }\Tilde{A}_{n+1,i}'\text{ by balls of radius }\beta^{n+1+i}r_0.
%\end{equation}

Let $\bx\in \cap_{n\geq 0} B_n$ be the outcome. By the definition of the game \ref{cantor game}, for every $n\geq 0,$ $B_{n+1}$ is disjoint from the sets in $\mathcal{A}_{n'+1,i'}$ where $n'+i'=n$, when $i'\neq i_0$ and disjoint from the sets in $\mathcal{A}'_{n'+1,i_0}$ where $n'+i_0=n$ by \eqref{eqn: strategy} . Thus 
%$\bx\notin \bigcup_{n\geq0,i\geq0}\bigcup_{A\in\mathcal A'_{n+1,i}}A \implies$
\begin{equation}\label{eqn: xnotin A'}
\bx\notin \bigcup_{n\geq1}\tilde A'_{n+1,i_0},
\end{equation} 
and 
\begin{equation}\label{eqn: xnotin union A}
    \bx\notin \bigcup_{n\geq0,i\geq0}\tilde A_{n+1,i}=\bigcup_{n\geq0,i\geq0}A_{n+1,i},
\end{equation}
since %the setting that $s-1\not|i_0$ and 
$\tilde A_{n+1,i_0}=\emptyset,\forall n\geq1.$ 

Next, we claim that for every $m\geq 1, 1\leq j\leq m^\star,$ and $\bv\in\mathcal{V}_m$, $\bx\not\in \Delta^{j,m}_{\ta}(\bv)$. If not, then %by \eqref{eqn: union}
$\bx\in \Delta^{j,m}_{\ta}(\bv)$ for some $\bv\in \mathcal{V}_m, 1\leq j\leq m^\star, m\geq 1$. Therefore, $\bx\in \tilde A'_{m+1,i_0}$ by \eqref{excluded} and \eqref{eqn: xnotin union A} as $\bx\in \cap_{n\geq 0} B_n$. This leads to a contradiction to \eqref{eqn: xnotin A'}. So we get $\bx\notin \bigcup_{n\geq 1}\bigcup_{\bv\in\mathcal{V}_n}\bigcup_{1\leq j\leq n^\star}\Delta^{j,n}_{\ta}(\bv)$. By \eqref{eqn: danger_bad} $\bx\in \mathbf{Bad}_{\ta}(\bw).$

%\begin{equation}
%\bx\in B_0\setminus\bigcup_{n\geq0,i\geq0}\bigcup_{A\in\mathcal A'_{n+1,i}}A\subset B_0\setminus\bigcup_{n\geq0}\tilde A'_{n+1,i_0}\stackrel{\ref{eqn: union},\ref{excluded}}{\subset} B_0\setminus\bigcup_{\bv\in\Q^d}\Delta_{\ta}(\bv)\stackrel{\ref{def: inho_interval}}{\subset}\mathbf{Bad}_{\ta}(\bw).
%\end{equation}

%The second half of the inequality above follows from the definition of $\mathbf{Bad}_{\ta}(\bw)$ and \cref{def: inho_interval}.

%\qed  %it suffices to show that the strategy is legal.

%For $n\geq 0,i\neq i_0$, we can simply let $\mathcal{A}_{n+1,i}'=\mathcal{A}_{n+1,i}$ where the latter cover exists as in \cref{borrowed} because \cref{containing} ensures that $\forall \bx\in B_n,\bx\notin\bigcup_{n'< n,i'\geq 0}\Tilde{A}_{n'+1,i'}'\supset\bigcup_{n'< n,i'\geq 0}\Tilde{A}_{n'+1,i'}$, thus $\mathcal{A}_{n+1,i}'$ is still an efficient as described in \cref{hom cover}.

 %The last step is to show the following lemma
 
 \subsection{Proof of Lemma \ref{cantorwinning}}\label{Proofcantorwin}
 %\sd{ One thing I am a bit confused about. In the statement of the following Lemma 3.4, we do not know what future $B_{n+1}$'s are in our game, and it is not used that $\bx\in B_n$ to deduce the conclusion. I think we are only using the fact that no matter what $B_n$'s are the strategy ensures $A_{n+1,i}$ is inside $\{x\in B_0: d_l a_{n+1+sl}u_{\bx}\notin K_{\beta^{\eta l}}\}$. And this is what we only use.}
 Denote
\begin{equation}\label{defn: D_n}
    D_{n}:=\{\bx\in B_0,d_1a_{n'+s}u_{\bx}\Z^{d+1}\in K_{\beta^{\eta},},\forall n'=1,...,n-1\},n\in\mathbb N.
\end{equation}
\begin{lemma}
\label{bounded}
    There exists $\xi>0$ such that for any $n\in\N,\forall ~
    %\bx\in B_n\setminus\bigcup_{0\leq n'<n,i'\geq0}A_{n'+1,i'}$,
    \bx\in D_{n}$, we have 
    \begin{equation}
        \{a_yu_{\bx}\Z^{d+1},0<y\leq n\}\subset K_{\xi}.
    \end{equation}
\end{lemma}
\begin{proof}
    Note that by the hypothesis,  %$\bx\notin A_{n',s-1}, \forall n'< n$, so\ls{To be removed} 
    $d_1a_{n'+s}u_{\bx}\Z^{d+1}\in K_{\beta^{\eta}}$. Therefore,
    \begin{equation}
    \label{finite arc}
        \forall n'\in\N\text{ with }0<n'< n,a_{n'}u_{\bx}\Z^{d+1}\in K_{\beta^{\eta+\frac{s}{1+w_1}+\frac{d+1-t}{d+1}}},
    \end{equation} where $t$ is as in $\eqref{defn: t}$.
    Note that the right hand side of \eqref{finite arc} is independent of $n$ and $\bx$, so by Mahler's criteria, the lemma follows. 
\end{proof}
Let $\xi$ be as in Lemma \ref{bounded}, and without loss of generality, we assume \begin{equation}
    \xi<\frac{d+1}{b(1+\mathfrak{m})}\label{xi}.
\end{equation}We define the following quantities:

\begin{equation}\label{def: lambda_1}\lambda_1:=\left\lceil \frac{1}{w_d}\log_b{\left(\frac{d+1}{\xi}\right)}\right\rceil,\end{equation} %\sd{ I think you mean $\log_b$ and not $\ln.$}
\begin{equation}
    k_1:=\frac{\xi}{d+1}b^{-\lambda_1}.
    \label{def: k_1}
\end{equation}
In the above definitions $\lambda_1$ and $k_1$ is dependent on $\beta$ since $\xi$ depends on $\beta.$
Also note from \eqref{def: lambda_1} and \eqref{order in w}, we have \begin{equation}\label{eqn: defn of lambda_1 gives}
b^{-\lambda_1 w_i}<\frac{\xi}{d+1}, 1\leq i\leq d.
\end{equation}
Also, \begin{equation}\label{upper bound on k_1r_1}
    k_1^{1+w_1}<\beta^{\lambda_1},
\end{equation}
since by \eqref{xi}, $\frac{d+1}{\xi}\stackrel{\eqref{def: k_1}}{=}k_1^{-1}b^{-\lambda_1}>1\implies k_1<b^{-\lambda_1}\stackrel{\eqref{beta}}{\implies} k_1^{1+w_1}<\beta^{\lambda_1}.$
Define $c$  to be
\begin{equation}
\label{def: c}
    c:=\frac{k_1^{1+w_1}\beta^4}{200d}.
\end{equation}
Note that 
\begin{equation}
\mathcal{V}_s=\emptyset, ~~\forall~~ 0\leq s\leq \lambda_1+1, s\in\N.\label{empty}\end{equation} 
This is because if $n\leq \lambda_1+1$, then $k^{1+w_1}<\beta^{\lambda_1}\leq\beta^{n-1}\stackrel{\eqref{def: c}}{\implies} c<\beta^{n+3}.$ Using \eqref{R_0'} this in particular implies $c r_0^{-1} \beta^{-(n+2)}<1$ which means by definition \eqref{def: V_n}, that $\mathcal{V}_n=\emptyset.$

%\sd{rewrote it see above} \gray{This is because if $\frac{\bp}{m}\in\mathcal{V}_n,$ \begin{equation}\label{upper bound on m}
%1\leq m\stackrel{\eqref{R_0'}, \eqref{def: c}}{<}\left(\frac{k_1^{1+w_1}}{d}\beta^{-(n-1)}\right)^{\frac{1}{1+w_1}}<b^{n-1-\lambda_1}.\end{equation} 

%By \cref{beta}, the last inequality holds if $\frac{k_1^{1+w_1}}{d } < \beta^{\lambda_1}$, which is true by \cref{upper bound on k_1r_1}. Now \cref{upper bound on m} implies for $q\leq \lambda_1+1$, there is no such $m$, hence $\mathcal{V}_n=\emptyset.$}

%In particular, $k_1<1$.\sd{ this is redundant line}

Then we have the following corollary.
%\sd{ Same remark as before, should not include $B_n$}
\begin{corollary}
\label{dual}
     For all $ n\in\N\text{ with }n>\lambda_1,~\forall ~1\leq H<b^{n-\lambda_1},~\forall 
     %\bx\in B_n$$\setminus\bigcup_{0\leq n'<n,i'\geq0}A_{n'+1,i'}$,\blue{
     \bx\in D_{n}$, there are no non-zero integer solutions $(z_0,z_1,\cdots,z_d)$ to \begin{equation}|z_0+\sum_{i=1}^d z_ix_i|\leq k_1H^{-1},|z_i|\leq H^{w_i},~ 1\leq i \leq d.\label{dual system}\end{equation}
\end{corollary}
\begin{proof}
    We prove this by contradiction. Let $(z_0,z_1,\cdots,z_d)$ be a non-zero solution to \eqref{dual system} with a positive number \begin{equation}1\leq H<b^{n-\lambda_1}\label{H},\end{equation} and %$\bx\in B_n$$\setminus\bigcup_{0\leq n'<n,i'\geq0}A_{n'+1,i'}$,\blue{
    $\bx\in D_{n}$. Then, we let 
    \begin{equation}
        \label{t_0}
        t_0:=\lambda_1+\log_b H,
    \end{equation}
    Thus $t_0<n.$
    %\sd{Minor: I think you wanted to take $ t_0:=\beta\lambda_1+\ln H$. Also, mention that because of this choice, and the range of $H,$ $t_0<\beta q. $}
    Also, let
    \begin{equation}
        \label{b}
        \mathbf{b}:=(b_i)_{i=0}^d=a_{t_0}u_{\bx}\cdot (z_0,z_1,\cdots,z_d)^T.
    \end{equation}

    Thus, we have $b_0=b^{\lambda_1}H|z_0+\sum_{i=1}^d z_ix_i|\stackrel{\eqref{dual system},\eqref{def: k_1}}{\leq}\frac{\xi}{d+1}$ and $\forall 1\leq i\leq d$, $b_i=b^{-\lambda_1w_i}H^{-w_i}|z_i|\stackrel{\eqref{dual system}, \eqref{eqn: defn of lambda_1 gives}}{\leq}\frac{\xi}{d+1}$, thus $\Vert \mathbf{b}\Vert_2=\sqrt{\sum_{i=0}^db_i^2}<\xi$. This implies that there exists $\bx\in D_n,$  and by \eqref{t_0} and \eqref{H}, some  $t_0< n$ such that 
    
    $$a_{t_0} u_{\bx}\Z^{d+1}\notin K_{\xi},$$
    
    which is a contradiction to Lemma \ref{bounded}.
\end{proof}

\begin{proposition}
\label{result}
$\forall n\in\N\text{ with }n>\lambda_1,%\bx\in B_n$$\setminus\bigcup_{0\leq n'<n,i'\geq0}A_{n'+1,i'},$\blue{$
\bx\in D_{n}$ and $1\leq Q<b^{n-\lambda_1}$,there will be no non-zero integer solutions $(m,p_1,\cdots,p_d)$ to 
\begin{equation}
\label{simul system}
    |m|\leq Q,|mx_i-p_i|<\frac{k_1}{d}Q^{-w_i},\forall 1\leq i\leq d.
\end{equation}
\end{proposition}
\begin{proof}
   Suppose by contradiction, $(m,p_1,\cdots,p_d)$ be a non-zero integer solution to \eqref{simul system} with a positive number $1\leq Q<b^{n-\lambda_1}$ and %$\bx\in B_n$$\setminus\bigcup_{0\leq n'<n,i'\geq0}A_{n'+1,i'}$,\blue{
   $\bx\in D_{n}$. Let $$L_0(m,p_1,\cdots,p_d):=m,L_i(m,p_1,\cdots,p_d):=mx_i-p_i,~\forall~ 1\leq i\leq d.$$ Now the transposed system is $L_0'(v_0,v_1,...,v_d)=v_0+\sum_{i=1}^dx_iv_i$ and $L_i'=-v_i,~\forall~ 1\leq i\leq d$. Then \eqref{simul system} implies \eqref{system1} with these $L_0,L_1,...,L_d$ and $T_0=Q,T_i=\frac{k_1}{d}Q^{-w_i},~\forall 1\leq i\leq d$. Now by Lemma \ref{transfer}, we have a non-zero integer solution $(v_0,v_1,\cdots,v_d)$ to \eqref{system2} with these $L_i', ~0\leq i\leq d$ and $\iota=\frac{k_1}{d}$. So $(v_0,v_1,...,v_d)$ as a non-zero integer solution to  \eqref{dual system} with $H=Q$, which is a contradiction to Corollary \ref{dual}.
\end{proof}

\begin{proposition}
    $\forall n\in\N,$ there is at most one $\bv\in\mathcal{V}_n$ such that $$\Delta_{\ta}(\bv)\cap B_n\setminus\bigcup_{0\leq n'<n,i'\geq0}A_{n'+1,i'}\neq\emptyset.$$\label{final lemma-part1}
\end{proposition}
\begin{proof}
If $n\leq\lambda_1,$ then $\mathcal{V}_n=\emptyset$ by \eqref{empty} and the proposition trivially holds. Now we assume that $n>\lambda_1.$ 

We prove by contradiction. Suppose there is $n\in\N$, and two such distinct $\bv_1=\frac{\bp_1}{m_1},\bv_2=\frac{\bp_2}{m_2}\in\mathcal{V}_n$, $\bp_{j}=(p_{i,j})_{i=1}^d,$ such that $\forall j=1,2,~ \exists~ \bx_j\in B_n\setminus\bigcup_{0\leq n'<n,i'\geq0}A_{n'+1,i'}$,
\begin{equation}
\label{inho_1}
    \left|x_{i,j}-\frac{p_{i,s}+\theta_{i}(x_{i,j})}{m_j}\right|<\frac{\mathfrak{q}c}{m_j^{1+w_i}}\stackrel{\eqref{q<1}}{<}\frac{c}{m_j^{1+w_i}},\forall~ 1\leq i\leq d.
\end{equation}

Without loss of generality,
\begin{equation}
    m_1\geq m_2. 
    \label{assume}
\end{equation}
Since $\bx_1,\bx_2\in B_n$$\setminus\bigcup_{0\leq n'<n,i'\geq0}A_{n'+1,i'}$, we have
\begin{equation}
    \forall~ 1\leq i\leq d, |x_{i,1}-x_{i,2}|<2r(B_n)=2r_0\beta^n.
    \label{diam}
\end{equation}
Then we have for any $1\leq i\leq d$,
\begin{equation}\label{ho1}
\begin{aligned}
    &|(m_1-m_2)x_{i,1}-(p_{i,1}-p_{i,2})|\\
    &\leq m_2|x_{i,1}-x_{i,2}|+\sum_{j=1,2}m_j\left |x_{i,j}-\frac{p_{i,j}+\theta_i(x_{i,j})}{m_j}\right|+\vert \theta_i(x_{i,1})-\theta_i(x_{i,2})\vert\\
    &\stackrel{\eqref{diam},\eqref{inho_1}}{<}2m_2r_0\beta^n+\sum_{j=1,2}\frac{c}{m_j^{w_i}}+2\mathfrak{m} r_0\beta^n\\
    &\stackrel{\eqref{assume}}{\leq}2(1+\mathfrak{m})m_2r_0\beta^n+\frac{2c}{m_2^{w_i}}.
\end{aligned}
\end{equation}
%\sd{$$
%\vert (\bq_1-\bq_2)\cdot\bx_1-p\vert \leq \Vert \bq_2\Vert\Vert \bx_1-\bx_2\Vert+\sum_{i=1,2}\vert\bq_i\cdot\bx_i+p_i+\theta\vert\leq b^{w_1n}\beta^{-n}+2b^{-n}\asymp b^{-n}
%$$}
Note we may assume $m_1>m_2$. Otherwise, as the right sides of \eqref{ho1} is $<1$,  $\bp_1=\bp_2$, giving $\bv_1=\bv_2,$ that contradicts our assumption.

Next, we claim that 
\begin{equation}\label{sol1}
  (m_1-m_2)^{w_i}|(m_1-m_2)x_{i,1}-(p_{i,1}-p_{i,2})|<\frac{k_1}{d}.
\end{equation}

This can be viewed as follows:
\begin{equation*}
\begin{aligned}
(m_1-m_2)^{w_i}|(m_1-m_2)x_{i,1}-(p_{i,1}-p_{i,2})|&\stackrel{\eqref{ho1},\eqref{assume}}{<}2(1+\mathfrak{m})m_1^{1+w_i}r_0\beta^n+\frac{2cm_1^{w_i}}{m_2^{w_i}}\\&\stackrel{\eqref{def: V_n}, \eqref{R_0'}}{<}2c\beta^{-3}+2c\beta^{-1}\stackrel{\eqref{def: c}}{<}\frac{k_1}{d}.
\end{aligned}
\label{abc1}
\end{equation*}

%$m_1^{1+w_1}< 2c|I_0|^{-1}R^{q+2}< 2cR^{q+3}< 2\frac{k_1}{10 n f_0} R^q$
We also note,
\begin{equation}m_1-m_2\leq m_1\stackrel{\eqref{R_0'}, \eqref{def: c}}{<}(\frac{k_1^{1+w_1}}{d}\beta^{-n})^{\frac{1}{1+w_1}}<b^ {-\lambda_1+n}.\end{equation}

By \eqref{beta}, the last inequality holds if $\frac{k_1^{1+w_1}}{d} < \beta^{\lambda_1}$, which follows from $\eqref{upper bound on k_1r_1}$ and the dimension $d\geq1.$ %This desired condition follows from $k_1\stackrel{\eqref{def: k_1}}{=}\frac{\xi}{d+1}b^{-\lambda_1}\stackrel{\eqref{beta}}{=} \frac{\xi}{d+1}\beta^{\frac{\lambda_1}{1+w_1}}\stackrel{\eqref{xi}}{<}\beta^{\frac{\lambda_1}{1+w_1}}\implies \frac{k_1^{1+w_1}}{d}\leq \frac{1}{d} \beta^{\lambda_1}. $ %Since $b^{\lambda_1 w_n}\stackrel{\eqref{def: lambda_1}}{\asymp} \frac{d+1}{\xi}\stackrel{\eqref{def: k_1}}{=}k_1^{-1}b^{-\lambda_1}\implies k_1\asymp b^{-\lambda_1(1+w_n)}= \left(\beta^{\lambda_1}\right)^{\frac{1+w_n}{1+w_1}}.$ Hence $\frac{k_1^{1+w_1}}{n}\ll \left(\beta^{\lambda_1}\right)^{1+w_n}<\beta^{\lambda_1}.$

%\sd{ The last line can be replaced by this (we don't need 3.29): By \eqref{def: k_1} $k_1=\frac{\xi}{d+1}b^{-\lambda_1}\stackrel{\eqref{beta}}{=} \frac{\xi}{d+1}\beta^{\frac{\lambda_1}{1+w_1}}\implies \frac{k_1^{1+w_1}}{d}\leq \frac{1}{d} (\frac{\xi}{d+1})^{1+w_1}\beta^{\lambda_1}. $ Now the conclusion follows using \eqref{xi}.}\ls{Fixed.}

 %\sd{Can you explain why the last inequality holds? We need $k_1/n< R^{-\lambda_1}$. But I think that is only true if $r_1=r_n.$ }

%\ls{I just made $c$ smaller. I think it should be fine now.}

We claim that $B_n\setminus\bigcup_{0\leq n'<n,i'\geq0}A_{n'+1,i'}\subset D_{n}.$ Let $\by\in B_n\setminus\bigcup_{0\leq n'<n,i'\geq0}A_{n'+1,i'},$ then we have $\by\in B_{n'},\forall~ 0<n'<n$ since the game requires that $B_n\subset B_{n-1}\subset...\subset B_0.$ Therefore for $1\leq n'<n,$ since $\by\notin A_{n'+1,s-1}$, $d_1a_{n'+1+s}u_\by\Z^{d+1}\in K_{\beta^\eta}\implies d_1a_{n'+s}u_\by\Z^{d+1}\in K_{\beta^\eta}$.%Therefore, for $1\leq n'<n,$ since $\by\notin A_{n'+1,s-1}$, it implies that $d_1a_{n'+s}u_\by\Z^{d+1}\in K_{\beta^\eta}$ from \eqref{def: original target set}. 
%\sd{Replace the last line: Therefore for $1\leq n'<n,$ since $\by\notin A_{n'+1,s-1}$, $d_1a_{n'+1+s}u_\by\Z^{d+1}\in K_{\beta^\eta}\implies d_1a_{n'+s}u_\by\Z^{d+1}\in K_{\beta^\eta}$. }
Thus from the definition of $D_n$ in \eqref{defn: D_n} the claim follows.

Since $m_1>m_2$, by \eqref{sol1}, the integer $(m_1-m_2,\bp_1-\bp_2)$ is a non-zero solution to the system \eqref{simul system} with $Q=m_1-m_2<b^{n-\lambda_1}$, with $\bx_1\in B_n\setminus\bigcup_{0\leq n'<n,i'\geq0}A_{n'+1,i'}\subset D_{n}$$,n>\lambda_1$, which leads to a contradiction to Proposition \ref{result}. Hence the proof of this proposition is complete.

\end{proof}
\begin{lemma}\label{lemma: unique j}
Suppose $\bv=\frac{\bp}{m}\in\mathcal{V}_n$,
%$$\Delta_{\ta}(\bv)\cap B_n\setminus\bigcup_{0\leq n'<n,i'\geq0}A_{n'+1,i'}\neq\emptyset.$$ 
then there exists upto $2^d$ many $1\leq j\leq n^\star$ such that 
$$\Delta_{\ta}^{j,n}(\bv)\neq\emptyset.$$

\end{lemma}

\begin{proof}
    Suppose $1\leq j_1<j_2\leq n^\star$ such that 
    $\by_\iota\in \Delta_{\ta}^{j_{\iota},n}(\bv), \iota=1, 2.$ $$\left|y_{i,\iota}-\frac{p_{i}+\theta_{i}(y_{i,\iota})}{m}\right|<\frac{\mathfrak{q}c}{m^{1+w_i}}, \quad i=1,\cdots, d.$$
Then for $i=1,\cdots, d$, $$\begin{aligned}\vert y_{i,1}-y_{i,2}\vert \leq &\sum_{\iota=1,2} \left\vert y_{i,\iota}- \frac{p_{i}+\theta_i(y_{i,\iota})}{m}\right\vert + \frac{1}{m}\vert \theta_i(y_{i,1})-\theta_i(y_{i,2})\vert\\
\leq& \frac{2\mathfrak{q}c}{m^{1+w_i}}+ \frac{\mathfrak{m}}{m} \vert y_{i,1}-y_{i,2}\vert.\end{aligned}$$
Thus for $i=1,\cdots, d$,
$$
\vert y_{i,1}-y_{i,2}\vert < \frac{2\mathfrak{q}c}{m^{w_i}(m-\mathfrak{m})}<\frac{\mathfrak{q}c (r_0 \beta c)}{ 2\mathfrak{m}(cr_0^{-1})^{w_i} \beta^{-\frac{n+1}{1+w_1}w_i}},
$$
since $m-\mathfrak{m}> 2c^{-1}\beta^{-1} r_0^{-1} \mathfrak{m}.$  In $B_0^{j,n}$ the side length in the $i$-th direction is $$
\frac{\mathfrak{q}r_0 c}{\mathfrak{m}\beta^{-\frac{(n+2)}{(1+w_1)}w_i}}>\frac{\mathfrak{q}c (r_0 \beta c)}{ 2\mathfrak{m}(cr_0^{-1})^{w_i} \beta^{-\frac{n+1}{1+w_1}w_i}}, \text{ since } c, \beta, r_0<1.$$ Thus for each direction $1\leq i\leq d$, there are atmost two choices of $1\leq j\leq n^\star$, which concludes the lemma.
\end{proof}

%We need $\mathfrak{q} \beta^{-\frac{(n+2)}{(1+w_1)}w_i}<(cr_0^{-1})^{w_i}\beta^{-\frac{n+1}{1+w_1}w_i}\iff \mathfrak{q}<\beta^{\frac{1}{1+w_1}w_i} (cr_0^{-1})^{w_i}.$ By \eqref{R_0'}, the last inequality is true if $\mathfrak{q}<\beta^{(5+\lambda_1)}/{500 d}$, just need to make it small enough.

\subsubsection{Finishing the proof of Lemma \ref{cantorwinning} }
Now without loss of generality there exists some $\bv\in\mathcal{V}_n$ such that $\Delta_{\ta}(\bv)$ intersects with $B_n\setminus\bigcup_{0\leq n'<n,i'\geq0}A_{n'+1,i'}$ where  $B_n=\supp\mu\cap B(\bx_n,\beta^nr_0)$, otherwise $ \tilde A'_{n+1,i_0}=\emptyset$ and there is nothing to prove.

We choose a covering of $\tilde{A}'_{n+1,i_0}$ of radius $\frac{1}{3}\beta^{\alpha(n+i_0+1)}r_0$ as in Lemma \ref{cover} to be $\mathcal{A}_{n+1,i_0}'$. Now, note that by Definition \eqref{excluded} $$\tilde A'_{n+1,i_0}\subset \bigcup_{\bv\in\mathcal{V}_n} \bigcup_{1\leq j\leq n^\star}\Delta_{\ta}^{j,n}(\bv) \cap B_n \setminus\bigcup_{0\leq n'<n,i'\geq0}A_{n'+1,i'}.$$ 
By Proposition \ref{final lemma-part1}, in the above union there exists an unique $\bv_n=\frac{\bp}{m}\in\mathcal{V}_n$, $\bp=(p_i)_{i=1}^d$, and by Lemma \ref{lemma: unique j} $J_n\subset\{1,\cdots, n^\star\}, \# J_n\leq 2^d$ such that 
$$\Delta_{\ta}^{j,n}(\bv_n)\cap B_n\setminus\bigcup_{0\leq n'<n,i'\geq0}A_{n'+1,i'}\neq\emptyset, j\in J_n.$$
Thus 
$$\tilde A'_{n+1,i_0}\subset  \bigcup_{ j\in J_n}\Delta_{\ta}^{j,n}(\bv_n) \cap B_n .$$ 

Hence by Lemma \ref{cover}, 
\begin{equation}\label{eqn: cover1}
\begin{aligned}  \#\mathcal{A}_{n+1,i_0}'& \leq \frac{A\mu(B(\bigcup_{ j\in J_n}\Delta_{\ta}^{j,n}(\bv_n)\cap B_n,\frac{1}{3}\beta^{(n+i_0+1)}r_0))}{\beta^{\alpha(n+i_0+1)}(\frac{1}{3}r_0)^{\alpha}}\\ & \leq\sum_{ j\in J_n}\frac{A\mu(B(\Delta_{\ta}^{j,n}(\bv_n)\cap B_n,\frac{1}{3}\beta^{(n+i_0+1)}r_0))}{\beta^{\alpha(n+i_0+1)}(\frac{1}{3}r_0)^{\alpha}}\\
&\stackrel{Lemma~\ref{compare dangerous domain}}{\leq}\sum_{ j\in J_n}\frac{A\mu(B(2\tilde\Delta_{\ta}^{j,n}(\bv_n)\cap B_n,\frac{1}{3}\beta^{(n+i_0+1)}r_0))}{\beta^{\alpha(n+i_0+1)}(\frac{1}{3}r_0)^{\alpha}}.
\end{aligned}\end{equation}
Now note by Lemma~\ref{compare dangerous domain} $$\tilde\Delta_{\ta}^{j,n}(\bv_n)\subset B\left(\left\{\frac{p_1+\theta_1^{j,n}}{m}\right\}\times \R^{d-1}, \frac{\mathfrak{q}c}{m^{1+w_1}}\right).$$
Furthermore for any $r>0$,
$$B(2\tilde\Delta_{\ta}^{j,n}(\bv_n)\cap B_n, r)\subset B\left(\left\{\frac{p_1+\theta_1^{j,n}}{m}\right\}\times \R^{d-1} , \frac{\mathfrak{2q}c}{m^{1+w_1}}+ r\right)\cap B(\bx_n,\beta^nr_0+r).$$
Since $\bv_n=\frac{\bp}{m}\in \mathcal{V}_n$, we get $m^{1+w_1}> cr_0^{-1} \beta^{-(n+1)}.$
Meanwhile, from \eqref{def: q} and \eqref{def: i_0},we have $\frac{1}{3}\beta^{n+i_0+1}r_0\leq \frac{1}{3}\beta^{n+(i_0+1)/2}(A^2D6^{2\delta+\alpha})^{-1/\delta}r_0<\mathfrak q \beta^nr_0.$ So, when $r=\frac{1}{3}\beta^{n+i_0+1}r_0,$     $$B(2\tilde\Delta_{\ta}^{j,n}(\bv_n)\cap B_n, \frac{1}{3}\beta^{n+i_0+1}r_0)\subset B\left(\left\{\frac{p_1+\theta_1^{j,n}}{m}\right\}\times \R^{d-1} , 4\mathfrak{q}\beta^n r_0\right)\cap B(\bx_n,2\beta^nr_0):=L.$$
Using Definition \ref{ahlfors} and Definition \ref{def: absolutely decaying}
$$ \mu (L)\leq D (2\mathfrak{q})^{\delta}\mu(B(\bx_n,2\beta^nr_0))\leq AD(2\mathfrak{q})^{\delta}(2\beta^nr_0)^\alpha .$$
Now we get the last term in \eqref{eqn: cover1} to be
$$\begin{aligned}
&\leq \#  J_n A^2 D(2\mathfrak{q})^{\delta}(2\beta^nr_0)^\alpha \frac{1}{\beta^{\alpha(n+i_0+1)}(\frac{1}{3}r_0)^{\alpha}}\\
& \leq 
2^d A^2 D 2^\delta 2^{\alpha} \mathfrak{q}^\delta 3^\alpha \beta^{-\alpha(i_0+1)}\\
&\stackrel{ Definition ~\eqref{def: q}}{\leq}\beta^{-(\alpha-\frac{\delta}{2})(i_0+1)}.
\end{aligned}
$$
Finally, we let $\gamma=\max\{0,\alpha-\frac{\delta}{2}\}$ and we finish the proof.
%is unique, then %since $s-1$ doesn't divide $i_0$, 

\section{Nulity of weighted inhomogeneous bad in a manifold: Theorem \ref{measure-zero-theorem}}\label{null-sec}
We prove a general version of Theorem \ref{measure-zero-theorem}, namely as follows:

\begin{theorem}\label{thm: general_manifold}
    Let $\ta\in\R^n$, $\bw$ be a weight such that atleast $d$ many coordinates of $\bw$ are $\max_{i=1}^nw_{i}.$ Then for any nondegenerate manifold $\mathcal{M}$ of $\R^n$ for which at almost every point the tangent is not parallel to any axis, almost every point is not in $\Bad_{\ta}(\bw).$
\end{theorem}

 We recall the definition of nondegeneracy first. Let $\f:U\subset\R^d\to \R^n$ be map, where $U$ is an open subset of $\R^d.$  We say $\f$ is $l$-nondegenerate at $\bx$ if there exists a neighbourhood of $\bx$ inside $U$ where $\f$ is $C^l$ and it has $n$ many linearly independent partial derivatives of order $\leq l$. Moreover, we call $\f$ to be nondegenerate if there exists $l$ such that $\f$ is $l$-nondegenerate for almost every (Lebesgue) $\bx$ in $U.$ By the assumption of Theorem \ref{thm: general_manifold}, we can assume without loss of generality, that $\mathcal{M}$ is parametrized by $\f$, which is nondegenerate at almost every point in $U.$ Moreover, we can choose small enough $U$ such that $\f$ is also nonsingular. Under the condition on the manifold in Theorem \ref{thm: general_manifold} while working with any weight $\bw$, changing parametrization if required, we can assume without loss of generality that \begin{equation}\label{maxlessmin} \quad \min_{1 \leq i \leq d} w_{i} \geq \max_{1 \leq j \leq m} w_{d+j}.
    \end{equation}
and that \begin{equation}\label{parametrize}
\f(x)=(\bx, f_1(\bx),\cdots, f_m(\bx)),\text{ where } m=n-d.
\end{equation}
 The reason why \eqref{maxlessmin} is nonrestrictive is because almost every $\bx$ tangent is not parallel to any axis due to assumption on $\f$. Also, by shrinking $U$ if necessary, we assume there exists $M\ge 1$
 \begin{equation}\label{condi2}
    \max_{1\leq k\leq m}\max_{1\leq i,j\leq d}\sup_{\bx\in U}\max \left\{\vert {\partial_i\mathit{f}_k}(\bx)\vert,\vert {\partial^2_{i,j}\mathit{f}_k}(\bx)\vert\right\} \leq M.
 \end{equation}

Thus in order to prove Theorem \ref{thm: general_manifold}, we will prove the following theorem.
 \begin{theorem}\label{thm: more general}
Let $\ta\in\R^n$, $\bw$ be a weight in $\R^n$ satisfying
\begin{equation}\label{condition on w}
    w_1=w_2=\cdots=w_d=\max_{i=1}^n w_i.
\end{equation} Let $\f: U\subset \R^d\to \R^n$ as in \eqref{parametrize} be nondegenerate. Then almost every $\bx\in U$, $\f(\bx)$ is not in $\Bad_{\ta}(\bw).$
\end{theorem}

For simplicity we denote $\max_{i=1}^n w_i:=w_{max}.$ Also we remark that the notations \S \ref{Sec3} and this section are independent of each other. For $\ta=(\theta_i)_{i=1}^n$, $\delta>0$, let $$\mathcal{W}_{\delta}:=\{\bx\in \R^n~|~ \max_{i=1}^n\vert qx_i-p_i-\theta_i\vert^{1/w_i}<\frac{\delta}{\vert q\vert} \text{ for infinitely many } (q, \bp)\in\Z^{n+1}\}.
$$

It follows from the definition that, $$\mathbf f^{-1}(\mathbf{Bad}_{\ta}(\bw))= U\setminus \bigcap_{\delta>0} \f^{-1}\mathcal{W}_{\delta}.$$

 %To prove Theorem \ref{measure-zero-theorem}, it is enough to show that $$\mathcal L_d(\mathbf f^{-1}(\mathbf{Bad}_{\ta}(\bw))\cap B_0)=0,$$ where $B_0$ is a small neighborhood around $\bx_0.$

Let \begin{equation}
    g_t= \diag\left\{e^{w_nt}, \cdots, e^{w_{d+1}t}, e^{w_d t}, \cdots,  e^{w_1t}, e^{-t}\right\},
\end{equation}

and for $k\in\N$ and $1>c>0$,\begin{equation}
    g_{c,k}:=\diag\{\underbrace{c^{-1}k, \cdots, c^{-1}k}_{m}, \underbrace{c^{(m+1)/d}/k^{m/d},\cdots, c^{(m+1)/d}/k^{m/d}}_{d}, c^{-1}\}.
\end{equation}
For any matrix $g$ in $\mathrm{GL}_{n+1}(\R)$, we recall the definition of dual matrix $g^\star=\sigma_{n+1} \left(g^{T}\right)^{-1} \sigma_{n+1},$ where $\sigma_{k}$ is the long Weyl element in $\mathrm{GL}_{k}(\R).$ Let $\lambda_{i}(g\Z^{n+1})$ denotes the $i$-th \textit{ successive minima} of the closed unit Euclidean ball centered at $\mathbf{0}$ with respect to lattice $g\Z^{n+1}$.

For any $\bx\in U$, we define $u_1(\bx)$ as
\begin{equation}
    u_1(\bx):=\begin{bmatrix}
        \mathrm{I}_m & -\sigma_m^{-1} J(\bx) \sigma_d & \sigma_m^{-1} h(\bx)^T\\
        0 & \mathrm{I}_d & \sigma_d^{-1} \bx^T\\
        0 & 0 & 1
    \end{bmatrix},
\end{equation}
where $h(\bx)= (f_1,\cdots,f_m)(\bx)-J(\bx)\bx^T,$ and $J(\bx)=[\partial_{j}f_i(\bx)]_{1\leq i\leq m, 1\leq j\leq d}.$ Let us define
\begin{equation}\label{defn: E}
   E_{t,k,c}:=\{\bx\in U~|~\lambda_1(g_{c,k}^\star  g_t^\star u_1(\bx)^\star\Z^{n+1})<c \}, {E}_{k, c}:= \bigcup_{s\in \N}\bigcap_{t\geq s} E_{t,k,c}, E:= \bigcup_{k\in\N}\bigcap_{c>0} E_{k, c}.
\end{equation}
\begin{lemma}\label{main_lemma_null}
Let $\bw$ be a weight satisfying \eqref{condition on w} and $k\in \N$. Let $\f:U\to\R^n$ be a map $\f(\bx)=(\bx, f_1(\bx), \cdots, f_{m}(\bx))$, which is $l$-nondegenerate at $\bx_0\in U$ and $U$ is an open set in $\R^d$. Then there exists  $c(k)>0$ and a neighborhood $B_0$ around $\bx_0$ such that for every $B\subset B_0,$ there exists $K>0$ such that for every $c\leq c(k),$

$$\mathcal{L}_d(E_{k,c}\cap B)\leq K \left( (2M)^{d-1}k^{2d-1} c^{\frac{n+1}{d}}\right)^{\alpha/n+1}
 \mathcal{L}_d(B), $$ 
where $\alpha= \frac{1}{d(2l-1)}$, where $M$ is as in \eqref{condi2}.
\end{lemma}
\begin{proof}
We choose $c(k)$ such that for all $c\leq c(k),$ \begin{equation}\label{eq: c(k)}
(2Mm)^{d-1}k^{n-1}c c^{\frac{m+1}{d}}<1.
\end{equation}
From now on, $c\leq c(k)$.
First note that $$g_{c,k}^\star=\diag\{c, \underbrace{c^{-\frac{m+1}{d}}k^{m/d}, \cdots, c^{-\frac{m+1}{d}} k^{m/d} }_{d}, \underbrace{\frac{c}{k} ,\cdots,  \frac{c}{k}}_{m}, \}$$ and 
for $t\in\N$
$$
g_{t}^\star=\diag\{e^{t}, e^{-w_1 t}, \cdots, e^{-w_d t}, e^{-w_{d+1}t},\cdots, e^{-w_nt}\}.
$$
Also note that \begin{equation}\label{eqn:dual u_1(x)}
        u_1^\star(\bx)=\begin{bmatrix}
            1 & -\bx & -f(\bx)\\
            0 & \mathrm{I}_d & J(\bx)\\
            0 & 0   & \mathrm{I}_m
        \end{bmatrix}, f(\bx)=(f_i(\bx))_{i=1}^m.
    \end{equation}
Let us also choose $t_0$ such that for all $t>t_0=t_0(c,k)$
\begin{equation}\label{eq:t_0}
    e^{-t}< (2Mm)^{d-1}k^{n-1}c c^{\frac{m+1}{d}}<1.
\end{equation}

   %\sd{Remove: Suppose $s\in\N $ is such that $\bx\in \bigcap_{t\geq s}E_{t, k, c}$ then for all $t\geq s$}, 
   Suppose for some $t>t_0$, if $\bx\in E_{t,k,c}$ then there exists nonzero $(\ba, a_0)\in\Z^{n+1}$ such that 
   \begin{equation}\begin{aligned}
       &\vert a_0+\ba\cdot\f(\bx)\vert \leq  e^{-t}\\
       & \Vert \nabla f(\bx)\ba^T\Vert \leq c c^{\frac{m+1}{d}} \frac{1}{k^{m/d}}e^{w_{max} t} \\
       & \vert a_i\vert \leq k e^{w_i t}, d+1\leq i\leq n, \ba=(a_i)_{i=1}^n.
  \end{aligned} 
  \end{equation}
The second line follows from \eqref{condition on w} as we have that $w_{max}=w_1=w_2=\cdots=w_d.$ Using the last two inequalities in the system of inequalities above, we get that for all large $t$
$$\vert a_i\vert \leq \left( c c^{\frac{m+1}{d}}\frac{1}{k^{m/d}} e^{w_{max} t}+ Mm k \max_{j=1}^m e^{w_{d+j} t}\right)\leq 2Mm k e^{w_{max} t} , i=1,\cdots, d.$$%\ls{Do we need to replace $M k \max_{j=1}^m e^{w_{d+j} t}$ by $mM k \max_{j=1}^m e^{w_{d+j} t}$? Since there are $m$ terms of the form $a_i\partial _if_{j}(\bx),d+1\leq i\leq n$ in the $j$-th entry of $u_1^*(\bx)(a_0,\ba)^T.$ The bounds afterwards in the proof might need adjustment as well. }\sd{Done.}
 Note that since $ w_{max} d+\sum_{j=1}^m w_{d+j}=1,$
 
 $$   e^{-t} \left(c c^{\frac{m+1}{d}} \frac{1}{k^{m/d}}e^{w_{max} t}\right)(2M mk e^{w_{max} t})^{d-1} k^m e^{\sum_{i=d+1}^n w_i t}\leq (2Mm)^{d-1}k^{n-1}c c^{\frac{m+1}{d}}.$$

Now using \cite[Theorem 1.4]{BKM}, with $\delta=e^{-t}, K= \left(c c^{\frac{m+1}{d}} \frac{1}{k^{m/d}}e^{w_{max} t}\right), T_1=\cdots=T_d=2M mk e^{w_{max} t}, T_i=k T^{w_it}, d+1\leq i\leq n$, by \eqref{eq:t_0} we get for every $t>t_0$,
\begin{equation}\label{eqn: measure_every t}
\mathcal{L}_d(E_{t,k,c}\cap B)\ll \left((2Mm)^{d-1}k^{n-1} c^{\frac{n+1}{d}}\right)^{\alpha/n+1} \mathcal{L}_d(B),
\end{equation}
where $\alpha= \frac{1}{d(2l-1)}.$ %Now for any $s\in\N$, we choose $t_s>\max\{t_0,s\}$, large enough such that $\blue{\left(\frac{1}{t^2}\right)^{\alpha/n+1}}<\frac{1}{s^2}$ for all $t>t_s.$ 
Then we get from \eqref{eqn: measure_every t} for every $s\in\N,$
$$
\mathcal{L}_d(\bigcap_{t\geq s} E_{t,k,c}\cap B)\ll \left((2Mm)^{d-1}k^{n-1} c^{\frac{n+1}{d}}\right)^{\alpha/n+1} \mathcal{L}_d(B).
$$

Since the right hand side is independent of $s$, and $\bigcap_{t\geq s} E_{t,k,c}$ is increasing as $s\in\N$ increases, 
$$
\mathcal{L}_d(\bigcup_{s\in\N}\bigcap_{t\geq s} E_{t,k,c}\cap B)=\lim_{s\to\infty}\mathcal{L}_d(\bigcap_{t\geq s} E_{t,k,c}\cap B)\ll \left((2Mm)^{d-1}k^{n-1} c^{\frac{n+1}{d}}\right)^{\alpha/n+1} \mathcal{L}_d(B).
$$

Now the conclusion follows from the definition of $E_{k,c}$.

%\sd{remove everything after this}
%Now using \cite[Theorem 1.4]{BKM}, %we get that the measure of the above set is $$\ll \left( \delta^{m-d} e^{-t} c^{d} c^{d\frac{m+1}{d}}\delta^{d} e^{w_{max}d t} \delta^{-m} e^{\sum_{j=1}^m w_{d+j}t}\right)^{\alpha/n+1}.$$
 %we get for every $s\in\N$,
%$$
%\mathcal{L}_d(\bigcap_{t\geq s} E_{t,k,c}\cap B)\ll \left((2Mm)^{d-1}k^{n-1} c^{\frac{n+1}{d}}\right)^{\alpha/n+1} \mathcal{L}_d(B),
%$$
%where $\alpha= \frac{1}{d(2l-1)}.$
%\sd{There is an issue here!}

%\sd{Check again that the constant in the above $\ll$ is independent of $k, s$} 
%Since the right hand side is independent of $s$, the conclusion follows.

%$$\ll \left( \delta^{m-d} c^{d} c^{m+1}\delta^{d}  \delta^{-m}\right)^{\alpha/(n+1)}= c^{\alpha}.$$
   
\end{proof}

\begin{lemma}\label{nbhd-measure-zero}
Let us assume the hypothesis of the previous lemma. There exists a ball $B_0$ around $\bx_0$ such that for every $B\subset B_0$
$$\mathcal{L}_d(E\cap B)=0.$$
\end{lemma}
\begin{proof}
    Fix $\varepsilon>0$. For every $k\in\N$ there exists $c_k>0, c_k\leq c(k)$ such that $$\mathcal{L}_d(E_{k, c_k}\cap B)\ll_{B}\frac{\varepsilon}{k^2}\implies \mathcal{L}_d(\bigcup_{k\in\N} \bigcap_{c>0} E_{k, c}\cap B)\ll_{B} \sum \frac{\varepsilon}{k^2}\ll_{B} \varepsilon.$$ We can take $\varepsilon$ arbitrarily small, which gives us the conclusion.
\end{proof}

The complement of the set $E$ is 
\begin{equation}\label{def: F}
   \mathcal{F}:=  \bigcap_{k\in\N}\bigcup_{c>0} \{\bx\in U~|~\lambda_1(g_{c,k}^\star  g_t^\star u_1(\bx)^\star\Z^{n+1})\geq c, \text{ for infinitely many } t\in \N\}.
\end{equation}

Using \cite[Lemma 3.3]{BY} the above set is same as 

\begin{equation}
   \bigcap_{k\in\N}\bigcup_{c>0} \{\bx\in U~|~\lambda_{n+1}(g_{c,k}  g_t u_1(\bx)\Z^{n+1})\ll_{n}c^{-1}, \text{ for infinitely many } t\in \N\}.
\end{equation}
%\sd{Rewrite the statement of the lemma, what is $V?$ Is the neighbourhood of $\bx$ is denoted by $V_{\bx}$}\ls{fixed.}
\begin{lemma}\label{lindelof}
    If a $\mathcal L_d$-measurable set $X\subset \mathbb R^d$ %\sd{What does it mean?} 
    satisfies that for any $\bx\in X,$ there exists an open neighborhood $V_\bx$ in $\R^d$ with $\bx\in V_{\bx}$, such that $\mathcal L_d(X\cap V_{\bx})=0,$ then $\mathcal L_d(X)=0$.
\end{lemma}
\begin{proof}
    Notice that $\{X\cap V_{\bx},\bx\in X\}$ is an open cover of $X,$ then since $X$ as a subset of $\R^d$ is Lindelöf, there exists a countable subcover, say, $\{X\cap V_i,i\in\mathbb N\}.$ Then by assumption, $\forall i\in\mathbb N, \mathcal L_d(X\cap V_i)=0,$ thus $\mathcal L_d(X)=\mathcal L_d(\bigcup_{i\in\mathbb N}(X\cap V_i))\leq\sum_{i\in\mathbb N}\mathcal L_d(X\cap V_i)=0,$ which finishes the proof.
\end{proof}
\begin{corollary}\label{F is full} Let $E$ be defined as in \eqref{defn: E}. Then
    $$\mathcal L_d(E)=0.$$
\end{corollary}
\begin{proof}
    Let $N=\{\bx_0\in U:\mathbf f$ is nondegenerate at $\bx_0\}$, which is of full measure. Then $\mathcal L_d(E)=\mathcal L_d(E\cap N).$ Therefore, applying Lemma \ref{nbhd-measure-zero} with $B=B_0$ and Lemma \ref{lindelof} with $X=E\cap N$, we conclude that $\mathcal L_d(E)=\mathcal L_d(E\cap N)=0.$
\end{proof}
Next, we have the main proposition of this section. 
\begin{proposition}\label{main_prp_null}
    Let $\ta\in\R^n\setminus \{\mathbf{0}\}$, $\f: U\subset\R^d\to\R^n$ be a $C^2$ map satisfying \eqref{parametrize} and $\bw$ be a weight satisfying \eqref{maxlessmin}. Suppose $\bx\in \mathcal{F}$ and $\bx$ is not of the form $\left(\frac{p_i+\theta_i}{q}\right)_{i=1}^d$ for some $(\bp,q)\in\Z^{n+1}$. Then for every $s\in\N$ there exists $c>0$ (that depends on $\bx$) such that for infinitely many $(\bp, q)\in\Z^{n+1}$ such that 
    \begin{equation}
        \begin{aligned}
            & \left\vert x_i-\frac{p_i+\theta_i}{q}\right\vert\ll_n \frac{  c^{-1} c^{-(m+1)/d}s}{q^{w_i+1}}, ~ i=1,\cdots,d,\\
            & \left\vert f_{j}\left(\frac{p_1+\theta_1}{q},\cdots, \frac{p_d+\theta_d}{q}\right)- \frac{p_{d+j}+\theta_{d+j}}{q}\right\vert \leq \frac{1 }{sq^{w_{d+j}+1}}, ~ j=1,\cdots, m.
        \end{aligned}
    \end{equation}

    \end{proposition}
\begin{proof}
%By our assumption $w_{max}:=w_1=\cdots= w_d\geq \max_{j=1}^m w_{d+j}$. 
Suppose $\bx\in \mathcal{F},$ then for every $s\in\N$ there exists $c>0$ such that for infinitely many $t_i\in\N$, 
$$\lambda_{n+1}(g_{c,s}  g_{t_i} u_1(\bx)\Z^{n+1})\ll_{n}c^{-1}.$$

This implies there are $n+1$ linearly independent vectors 
$\bv_k= g_{c,s} g_{t_i} u_1(\bx)\bmi_k^T, \bmi_k\in\Z^{n+1}\setminus \{\mathbf{0\}}, 1\leq k\leq n+1$ in the a ball of radius $\asymp_n c^{-1}$ around $\mathbf{0}.$
Thus there are $\gamma_k\in\R$ such that 
$$
g_{c,s} g_{t_i}\widetilde u_1(\bx)(\ta', 0)=\sum_{k=1}^{n+1} \gamma_{k} \bv_k,
$$ where $$\widetilde u_1(\bx):=\begin{bmatrix}
        \mathrm{I}_m & -\sigma_m^{-1} J(\bx) \sigma_d & 0\\
        0 & \mathrm{I}_d & 0\\
        0 & 0 & 1
    \end{bmatrix}, \ta'= \sigma_{n} \ta.$$
Now suppose $\gamma_k'$ is the nearest integer of $\gamma_k$, then
$$
\left\Vert \sum_{k=1}^{n+1} \gamma'_{k} \bv_k-g_{c,s} g_{t_i}\widetilde u_1(\bx)(\ta',0)\right\Vert\ll_{n} c^{-1}.
$$
We claim that $\gamma_k'\neq 0$ for some $k.$ %\ls{This is for some $k$, right? But it doens't change the result any way}\sd{Yes, you are right. fixed} 
Otherwise, $\Vert g_{c,s} g_{t_i}\widetilde u_1(\bx)(\ta',0)\Vert\ll_n c^{-1}\implies \Vert g_{t_i}\widetilde u_1(\bx)(\ta',0)\Vert\ll_{c,s, n} 1 $ as $t_i\to\infty.$ This implies $\ta=\mathbf{0}$, which is a contradiction to our assumption.

For ease of notation, we call $\{t_i\}=\mathcal{T}$. From above, for $t\in\mathcal{T},$
there exists $(q, \bp)\in\Z^{n+1}\setminus \{\mathbf{0}\}$ such that 
\begin{equation}\label{eqn: a1}
    \begin{aligned}
        &\vert qf_j(\bx)-\sum_{i=1}^d \partial_if_{j}(\bx)(qx_i-p_i-\theta_i)-p_{d+j}-\theta_{d+j}\vert \ll_n \frac{1}{se^{w_{d+j}t}}, 1\leq j\leq m\\
        & \vert qx_i-p_i-\theta_i\vert \ll_{n} c^{-1} c^{-(m+1)/d}\frac{s^{m/d}}{e^{w_it}} , 1\leq i\leq d,\\
        & q\ll_n  e^t.
    \end{aligned}
\end{equation}
Now note that, 
$$\begin{aligned}
& \left| f_{j} \left(\frac{p_{1}+\theta_1}{q}, \dots , \frac{p_{d}+\theta_d}{q} \right)-f_{j}(\bx)-\sum_{i=1}^d \partial_i f_{j}(\bx) \left(\frac{p_{i}+\theta_i}{q}-x_{i} \right) \right| \\ & < C_{\bx}\max_{1 \leq i \leq d} \left| \frac{p_{i}+\theta_i}{q}-x_{i} \right|^{2}\\
& \stackrel{\eqref{eqn: a1}}{\ll_n} c^{-2} c^{-2(m+1)/d}s^{2m/d} q^{-2} \max_{1\leq i\leq d}\frac{1}{e^{2w_it}},\\
& \stackrel{\eqref{condition on w}}{\ll_n} c^{-2} c^{-2(m+1)/d}s^{2m/d} \frac{1}{q e^{w_{max} t}}\frac{1}{e^{w_{d+j}t} q}, \text{ since } q\geq 1.
\end{aligned}
$$
Since $\mathcal{T}$ is infinite, we get that infinitely many $t\in \mathcal{T}'\subset \mathcal{T}$, we have 
$$\left| f_{j} \left(\frac{p_{1}+\theta_1}{q}, \dots , \frac{p_{d}+\theta_d}{q} \right)-f_{j}(\bx)-\sum_{i=1}^d \partial_i f_{j}(\bx) \left(\frac{p_{i}+\theta_j}{q}-x_{i} \right) \right|< \frac{1}{se^{w_{d+j}t}q}.$$
Hence we get that 
$$\begin{aligned}
&\left\vert  qf_{j} \left(\frac{p_{1}+\theta_1}{q}, \dots , \frac{p_{d}+\theta_d}{q} \right) -p_{d+j}-\theta_{d+j}\right\vert <\frac{1}{se^{w_{d+j}t}},\\
& \vert q x_i-p_i-\theta_i\vert\ll_{n}  c^{-1} c^{-(m+1)/d}s^{m/d}\frac{1}{e^{w_it}}, 1\leq i\leq d,\\
& q\ll_n e^t.
\end{aligned}$$
Since $\bx$ is not of the form $(\frac{p_i+\theta_i}{q})_{i=1}^d,$ and $\mathcal{T}'$ is infinite, the proposition follows.
\end{proof}

\begin{remark}
    Note that in the previous proposition $\ta\neq \mathbf{0}$ is not necessary. One can prove even a stronger statement in case of $\ta=\mathbf{0}$; see \cite[Theorem 10.1]{BDGW_null}.
\end{remark}
Next, we have the following corollary:
\begin{corollary}\label{full-in-F}
   Let $\ta\in\R^n\setminus \{\mathbf{0}\}$, $\f: U\subset\R^d\to\R^n$ be a $C^2$ map satisfying \eqref{parametrize} and $\bw$ be a weight satisfying \eqref{condition on w}. Let $\delta_1,\delta_2>0$ be any two constants. Then for almost every $\bx\in U,$ 
there are infinitely many $(\bp,q)\in\Z^{n+1}$ such that 

\begin{equation}
     \begin{aligned}
            & \left\vert x_i-\frac{p_i+\theta_i}{q}\right\vert\leq \frac{\delta_1}{q^{w_i+1}}, i=1,\cdots, d,\\
            & \left\vert f_{j}\left(\frac{p_1+\theta_1}{q},\cdots, \frac{p_d+\theta_d}{q}\right)- \frac{p_{d+j}+\theta_{d+j}}{q}\right\vert \leq \frac{\delta_2 }{q^{w_{d+j}+1}}, j=1,\cdots,m.
        \end{aligned}
    \end{equation}
\begin{proof}
For $\delta_2>0$, suppose $s_2\in\N$ such that $\frac{1}{s_2}<\delta_2.$%\ls{Do we want $s_2$ instead of $s$?}\sd{Done} 
Now let us define,
$$
S(\delta_2):=\left\{(\bp,q)\in\Z^{n+1}~\left|~\begin{array}{l}(\frac{p_1+\theta_1}{q},\cdots, \frac{p_d+\theta_d}{q})\in U,\\
\left\vert f_{j}\left(\frac{p_1+\theta_1}{q},\cdots, \frac{p_d+\theta_d}{q}\right)- \frac{p_{d+j}+\theta_{d+j}}{q}\right\vert \leq \frac{\delta_2 }{q^{w_{d+j}+1}}, 1\leq j\leq m
\end{array}\right.\right\}.$$
Then for $(\bp,q)\in \Z^{n}\times \Z\setminus\{0\}$ and $\delta_1>0$, let us denote  
$$
B(\delta_1, (\bp,q)):=\prod_{i=1}^dB\left( \frac{p_i+\theta_i}{q}, \frac{\delta_1}{q^{w_{i}+1}}\right).
$$By Corollary \ref{F is full}, $\mathcal{L}_d(\mathcal{F})=\mathcal{L}_d(U).$%\ls{Starting from here, in my understanding, the logic should be:\\ By Proposition \ref{main_prp_null} there exists $\tilde \delta$ such that $$\mathcal{L}_d\left(\limsup_{(\bp,q)\in S(\delta_2)}B(\tilde\delta, (\bp,q))\right)=\mathcal{L}_d\left(\limsup_{(\bp,q)\in S(1/s_2)}B(\tilde\delta, (\bp,q))\cap U\right)=\mathcal L_d(U)$$(the first equal sign follows from definition of $S(1/s_2)$ and might be worthy for a lemma, \sd{ See in Proposition 4.1 the constant depends on $\bx$, so over a ball one can not guarantee that the same constant will work, see below}), then by Lemma \ref{CI_product} and $\delta_2>\frac{1}{s_2}$, we have \begin{equation}\begin{split}\mathcal{L}_d\left(\limsup_{(\bp,q)\in S(\delta_2)}B(\delta_1, (\bp,q))\cap U\right)&=\mathcal{L}_d\left(\limsup_{(\bp,q)\in S(\delta_2)}B(\delta_1, (\bp,q))\right)\\&\geq\mathcal{L}_d\left(\limsup_{(\bp,q)\in S(1/s_2)}B(\delta_1, (\bp,q))\right)\\&=\mathcal{L}_d\left(\limsup_{(\bp,q)\in S(1/s_2)}B(\tilde\delta, (\bp,q))\right)=\mathcal L_d(U).\end{split}\end{equation}} 
Now, by  Proposition \ref{main_prp_null} we get $$\begin{aligned}
\lim_{\delta\to +\infty} \mathcal{L}_d\left(\limsup_{(\bp,q)\in S(1/s_2)}B(\delta, (\bp,q))\cap U\right)=&\mathcal{L}_d\left(\bigcup_{\delta>0} \limsup_{(\bp,q)\in S(1/s_2)}B(\delta, (\bp,q))\cap U\right)\\= &\mathcal{L}_d(U).
\end{aligned}$$
%\ls{The thing confusing me is, why taking $\delta\to0$?, since $\mathcal{L}_d\left(\limsup_{(\bp,q)\in S(1/s_2)}B(\delta, (\bp,q))\cap U\right)$ shrinks as $\delta$ decreases. Do we want to take $\delta\to+\infty$?}
The first equality above is because the measure of the union of increasing sets is the limit of each set's measure and the second equality follows from Proposition \ref{main_prp_null}.

Now by Lemma \ref{CI_product}, for any $\delta>0$, $$\mathcal{L}_d\left(\limsup_{(\bp,q)\in S(1/s_2)}B(\delta, (\bp,q))\cap U\right)=\mathcal{L}_d\left(\limsup_{(\bp,q)\in S(1/s_2)}B(\delta_1, (\bp,q))\cap U\right).$$ By taking a limit $\delta\to +\infty,$ we get 
 $$
 \mathcal{L}_d(U)= \mathcal{L}_d\left(\limsup_{(\bp,q)\in S(1/s_2)}B(\delta_1, (\bp,q))\cap U\right).
 $$
 Since $\delta_2>1/s_2$ combining the above, we get that
$$\mathcal{L}_d\left(\limsup_{(\bp,q)\in S(\delta_2)}B(\delta_1, (\bp,q))\cap U\right)=\mathcal{L}_d(U).$$
%\sd{see if this explanation is satisfactory}
%\sd{ Check again that Lemma \ref{CI_product} actually gives the above.}\ls{See above}
\end{proof}
\end{corollary}

\begin{lemma}\label{inside_coro} There exists $C>0$ that only depends on $\f$ and $U$ such that for any $\delta>0,$
    $$\limsup_{(\bp,q)\in S(\delta/2)}B(C^{-1}\delta/2, (\bp,q))\cap U\subset \f^{-1}\mathcal{W}_{\delta}\subset \limsup_{(\bp,q)\in S(C\delta)}B(\delta, (\bp,q))\cap U. 
$$
\end{lemma}
%\ls{Why do we call it corollary?}\sd{ you are right, we should may be call it a Lemma. Done}
\begin{proof}
Let us fix $C=\max\{2(Mm+1)+1, Md(1+d)+1\},$ where $M$ is as in \eqref{condi2}.  
    Suppose $\bx\in \f^{-1}\mathcal{W}_{\delta}$, then there are infinitely many $(\bp,q)\in\Z^{n+1}$ such that 
$$\begin{aligned}
    &\left\vert x_i-\frac{p_i+\theta_i}{q}\right\vert<\frac{\delta}{\vert q\vert^{w_i+1}}, \quad i=1,\cdots, d,\\
    &\left\vert f_{j}(\bx)-\frac{p_{d+j}+\theta_{d+j}}{q}\right\vert \leq \frac{\delta}{\vert q\vert^{w_{d+j}+1}}, \quad j=1,\cdots,m.
\end{aligned}$$

Thus by Taylor's expansion and using \eqref{condi2}, 
$$\begin{aligned}
&\left\vert f_{j}\left(\frac{p_1+\theta_1}{q},\cdots, \frac{p_d+\theta_d}{q} \right)-\frac{p_{d+j}+\theta_{d+j}}{q}\right\vert\\
&\leq \frac{\delta Mm}{\vert q\vert^{w_{max}+1}}+\frac{\delta}{\vert q\vert^{w_{d+j}+1}}, \quad j=1,\cdots,m.
\end{aligned}$$
This proves the right-hand side of inclusion in the corollary by since $C> 2 (Mm+1)$.

Now suppose $\bx\in \limsup_{(\bp,q)\in S(\delta)}B(C^{-1}\delta, (\bp,q))\cap U,$ this means there are infinitely many $(\bp, q)\in \Z^{n+1}$ such that 
$$\begin{aligned}
    &\left\vert x_i-\frac{p_i+\theta_i}{q}\right\vert<\frac{C^{-1}\delta}{\vert q\vert^{w_i+1}}, \quad i=1,\cdots, d,\\
    &\left\vert f_{j}\left(\frac{p_1+\theta_1}{q},\cdots, \frac{p_d+\theta_d}{q} \right)-\frac{p_{d+j}+\theta_{d+j}}{q}\right\vert \leq \frac{\delta}{\vert q\vert^{w_{d+j}+1}}, \quad j=1,\cdots,m.
\end{aligned}$$
Note that for the above $(\bp, q)$
$$\begin{aligned}
&\left \vert f_{j}(\bx)-f_{j}\left(\frac{p_1+\theta_1}{q},\cdots, \frac{p_d+\theta_d}{q} \right)\right\vert \\
& \leq \sum_{i=1}^d\partial_if_{j}(\bx)\left\vert x_i-\frac{p_i+\theta_i}{q}\right\vert+ M d^2\max_{i=1}^d \left \vert x_i-\frac{p_i+\theta_i}{q}\right\vert^2\\  
& \stackrel{\eqref{condition on w}}{\leq} (Md+Md^2) \max_{i=1}^d\frac{C^{-1}\delta}{\vert q\vert^{w_i+1}} \\
& \stackrel{\eqref{maxlessmin}}{\leq} Md(1+d)C^{-1}\frac{\delta}{\vert q\vert^{w_{d+j}+1}},\quad j=1,\cdots,m.
\end{aligned}
$$
Since $Md(1+d)<C,$
we get 
 $$\left \vert f_{j}(\bx)- \frac{p_{d+j}+\theta_{d+j}}{q}\right\vert\leq \frac{2\delta}{\vert q\vert^{w_{d+j}+1} }, \quad j=1,\cdots, m.
 $$

\end{proof}

\subsection{Finishing the proof of Theorem \ref{measure-zero-theorem}}
Note that $\mathbf f^{-1}(\mathbf{Bad}_{\ta}(\bw))= U\setminus \bigcap_{\delta>0} \f^{-1}\mathcal{W}_{\delta}.$
    By Corollary \ref{full-in-F} and Lemma \ref{inside_coro}
    $\mathcal{L}_d\left(\f^{-1}\mathcal{W}_{\delta}\right)=\mathcal{L}_d(U).$
Thus $$\mathcal L_d(\mathbf f^{-1}(\mathbf{Bad}_{\ta}(\bw)))=0.$$

\subsection{Proof of Corollary \ref{analytic}}
Let $\mathcal{M}$ be an analytic nondegenerate manifold of dimension $d$ in $\R^n$. Without loss of generality, we can assume $\mathcal{M}$ to be an image of $\f=(f_i)_{i=1}^n:B\to \R^n$, where $B=(-t_0, t_0)\times \prod_{i=2}^d (-u_i, u_i)\subset \R^d$ is a rectangle centered at $0\in\R^d.$ Let us define $$W:=\left\{(t,u_2,\cdots, u_d)\in\R^d~\left|~\begin{aligned}
    &tu_2\cdots u_d\neq 0\\
    & (t^{1+s^d}, u_2 t^{s+s^d}, \cdots, u_d t^{s^{d-1}+s^d})\in B
\end{aligned}\right.\right\}$$ an open set. %\ls{In order that $\phi$ is a bijection, we need to exclude $\{0\}\times\R^{d-1}$ from $W$, otherwise whenever $\phi$ would send the input to 0 whenever $t=0$.}\sd{yes, you are right. Fixed} 
Also denote 
$$B^\star:=\{(x_1,\cdots, x_d)\in B~|~x_1x_2\cdots x_d\neq 0\}.$$ %\text{ if } x_1\neq 0\}.$$
%\ls{Same as above. We might need to replace this condition by $x_1...x_d\neq0$, which still is full measure.}\sd{Fixed}
Let $\phi: \R^d\to\R^d$ be 
$$\phi(t,\bu):=(t^{1+s^d}, u_2 t^{s+s^d}, \cdots, u_d t^{s^{d-1}+s^d}),$$ which is a bijection between $W$ and $B^\star.$ Take $\bu=(u_2,\cdots,u_d)$ such that there exists some $t$ such that $(t,u_2,\cdots,u_d)\in W$, and for such $\bu$, let us denote $W_{\bu}:=\{t~|~(t,\bu)\in W\}$.

%Take any $\bu=(u_1,\cdots, u_d)\in B.$ Now define, 
%$$\phi_{\bu}:(0,1)\to\R^d, \phi_{\bu}(t):=(u_1t^{1+s^d}, u_2 t^{s+s^d}, \cdots, u_d t^{s^{d-1}+s^d}).$$

%For every $\bu^0=(u_i^0)_{i=1}^d\in B$, we show that there exists an open ball $B_0\subset B$ such that  $\bu^0\in B_0$ and 
%$$\mathcal{L}_d(B_0\cap \f^{-1}\Bad_{\ta}(\bw))=0.$$ Now for every 
For every such $\bu$, 
let us define $\f_{\bu}:W_{\bu}\to \R^n$ where each coordinate map is defined as follows:
$$f_{\bu,i}(t):= f_i(t^{1+s^d}, u_2 t^{s+s^d}, \cdots, u_d t^{s^{d-1}+s^d}), 1\leq i\leq n.$$ 
By Fibering lemma \cite[p. 1206]{Beresnevich2015}, there exists $s>0$ such that for every $\bu$ under consideration, $1, f_{\bu,1},\cdots, f_{\bu,n}$ are linearly independent over $\R$. Since $\f_{\bu}$ is analytic and coordinate maps are linearly independent together with $1$, for  almost every $t\in W_{\bu}$, $\f_{\bu}$ is nondegenerate. Now using Theorem \ref{measure-zero-theorem} 
$$\mathcal{L}_1(\f_{\bu}^{-1} \Bad_{\ta}(\bw)\cap W_{\bu})=0.$$

Now using Fubini's theorem, we get that 
$$
\mathcal{L}_d\{(t,\bu)\in W ~|~\f_{\bu}(t)\in \Bad_{\ta}(\bw)\}=0.$$
Since $\f_{\bu}(t)=\f(\phi(t,\bu))$, $\phi$ is a bijection between $W$ and $B^\star$ we get that 
$$\mathcal{L}_d( \f^{-1}\Bad_{\ta}(\bw)\cap B^\star)=0.$$ We conclude as $\mathcal{L}_d(B)=\mathcal{L}_d(B^\star).$

\section{Future direction}

The following is a natural generalization of the result in this paper to the dual analogue of inhomogeneous badly approximable vectors.

For any constant $
\theta\in\R,$ we define the following set:
\begin{equation}
\label{def: dual_inho}
    \widetilde{\mathbf{Bad}_{\theta}(\mathbf{w})}:=\{\mathbf{x}\in\mathbb{R}^d:\liminf_{\bq=(q_i)\in\Z^d\setminus\{0\}, \Vert \bq\Vert\to\infty}\vert \bq\cdot \bx-\theta\vert_{\mathbb{Z}}  \max_{1\leq i\leq d}\vert q_i\vert^{1/w_i} >0 \}.
\end{equation} 
In \cite{Beresnevich2015}, it was shown that $\widetilde{\mathbf{Bad}_{0}(\mathbf{w})}=\mathbf{Bad}_{\mathbf{0}}(\mathbf{w}),$ using Khintchine’s transference principle. But in the inhomogeneous setting, the dual and simultaneous settings are not related. Therefore, it is interesting to consider the following problem.
\begin{problem}\label{dual problem}
    For any weight $\bw$ in $\R^d$, does $\widetilde{\mathbf{Bad}_{\theta}(\mathbf{w})}$ have HAW property? %Is $\widetilde{\mathbf{Bad}_{\theta}(\mathbf{w})}$ Lebesgue measure zero on nondegenerate manifolds in $\mathbb R^n$?
\end{problem}
 In the case of standard weight, we get affirmative answer by \cite{ET}. Thus extending the same for any weight is a charming direction. Also studying the differences of two badly approximable sets following \cite{DD24} in higher dimensions directs to many interesting questions.
 
As a continuation of our second main Theorem \ref{measure-zero-theorem} and Corollary \ref{analytic} we want to pose the following question. 
\begin{problem}\label{last qs}
   Let $\bw$ be a weight in $\R^n$, $\theta\in\R$, and $\ta\in\R^n$. Let $\mathcal{M}$ be a $C^2$ manifold in $\R^n$. Is it true that for almost every point in $\mathcal{M}$ is not in $\mathbf{Bad}_{\ta}(\bw)$? Is it true that for almost every point in $\mathcal{M}$ is not in $\widetilde{\mathbf{Bad}_{\theta}(\bw)}$?
\end{problem}

When $\ta=\mathbf\{\mathbf{0}\}$, then an affirmative answer to the above problem was shown in \cite{BDGW_null}. As pointed out before, and as seen in this paper, the problem is more difficult in the inhomogeneous case and some new ideas are required. We end this section with two more directions to explore.

Problem stated in \ref{prob_null} can be analogously asked for other subsets of $\R^n$ replacing submanifolds. For instance, the attractor of an \textit{iterated function system} intersected with $\Bad(1)$ was shown to be null with respect to some self-similar measures in $\R$ \cite{EFS}, see \cite{SW2019} for higher dimensions $\Bad(\frac{1}{d}, \cdots, \frac{1}{d})$, see \cite{AG24} for inhomogeneous analogue $\Bad_{\ta}(\frac{1}{d}, \cdots, \frac{1}{d})$. In the weighted set-up, a very recent work \cite{Khalil_LW25} shows for every weight $\bw$ in $\R^d$, $\Bad(\bw)$ is null intersected with support of certain self similar measures. Pursuing analogous question as in Problem \ref{last qs} in these \textit{fractal} set-up (and more generally, beyond the self similarity condition, for any Borel probability measure) seems quite natural and interesting.

\subsection*{Acknowledgements} The authors greatly acknowledge
the initial support from the University of Michigan where this work started. SD thanks Ralf Spatzier for his support and encouragement when this work was initiated. She is indebted to Victor Beresnevich for various insightful remarks and discussions that helped us immensely, especially Corollary \ref{analytic} was suggested by him. The authors also thank Anish Ghosh for his comments on an earlier draft of this preprint. SD was in part supported by the Knut and Alice Wallenberg Foundation and also by the EPSRC grant EP/Y016769/1.

\bibliographystyle{abbrv}
\bibliography{inhomo_curve.bib}

\end{document}